\documentclass[reqno,11pt]{amsart}
\usepackage[en-GB]{datetime2}
\usepackage{graphicx}
\usepackage{amssymb}
\usepackage{caption}
\usepackage{enumitem}
\setlist[itemize]{leftmargin=2.5em}
\setlist[enumerate]{leftmargin=2.5em,label=\upshape(\arabic*)}
\usepackage{comment}
\usepackage{thmtools}
\usepackage{zref-clever}
\zcsetup{cap=true,nameinlink=false}
\zcsetup{countertype={subsection    = subsection}}
\zcRefTypeSetup{subsection}{Name-sg=Subsection}
\newenvironment{clm}[1]
{\renewcommand\theclaim{#1}\claim}
{\endclaim}
\usepackage{setspace}
\usepackage[T1]{fontenc}
\usepackage[charter,cal=cmcal]{mathdesign}	
\usepackage[nomath]{stix2}				
\usepackage{etoolbox}
\makeatletter
\patchcmd{\section}{\scshape}{\fontfamily{stix2}\fontshape{sc}\selectfont}{}{}			
\patchcmd{\abstract}{\scshape}{\fontfamily{stix2}\fontshape{sc}\selectfont}{}{}			
\patchcmd{\@setaddresses}{\scshape}{\fontfamily{stix2}\fontshape{sc}\selectfont}{}{}			
\makeatother

\usepackage{tikz}
\usetikzlibrary{matrix,intersections,calc,shapes.geometric,patterns}
\tikzset{ 
	table/.style={
		matrix of nodes,
		nodes={rectangle,text width=1.75em,align=center},
		text depth=1.25ex,
		text height=2.5ex,
		nodes in empty cells
	}
}
\usepackage{mathtools}
\usepackage{amsfonts}
\usepackage{cite}

\usepackage[pagebackref
]{hyperref}
\hypersetup{
	colorlinks=true,
	allcolors=blue,
	pdftitle={Ramsey-type χ-bounds for  χ-bounded graph classes},
	pdfauthor={Tung Nguyen and Sang-il Oum}
}
\usepackage{geometry}
\newtheorem{theorem}{Theorem}[section]
\newtheorem{lemma}[theorem]{Lemma}
\newtheorem{claim}{Claim}[theorem]
\newtheorem{corollary}[theorem]{Corollary}
\newtheorem{conjecture}[theorem]{Conjecture}
\newtheorem{proposition}[theorem]{Proposition}
\zcRefTypeSetup{claim}{
	Name-sg = Claim ,
	name-sg = claim ,
	Name-pl = Claims ,
	name-pl = claims ,
}
\zcRefTypeSetup{conjecture}{
	Name-sg = Conjecture ,
	name-sg = conjecture ,
	Name-pl = Conjectures ,
	name-pl = Conjectures ,
}
\expandafter\let\expandafter\oldproof\csname\string\proof\endcsname
\let\oldendproof\endproof
\renewenvironment{proof}[1][\proofname]{%
	\oldproof[\normalfont\bfseries #1]%
}{\oldendproof}
\newenvironment{subproof}[1][\normalfont\it Subproof]{%
	\begin{proof}[#1]%
	}{%
	\end{proof}%
}

\newcommand{\mac}{\mathcal}
\newcommand{\mab}{\mathbb}

\newcommand{\eps}{\varepsilon}

\newcommand{\de}{\operatorname{d}}
\renewcommand{\subset}{\subseteq}

\DeclarePairedDelimiter\abs{\lvert}{\rvert}%
\DeclarePairedDelimiter\ceil{\lceil}{\rceil}%
\DeclarePairedDelimiter\floor{\lfloor}{\rfloor}%
\makeatletter
\newcommand{\leqnomode}{\tagsleft@true}
\newcommand{\reqnomode}{\tagsleft@false}
\makeatother

\begin{document}
	\title{
		Ramsey-type $\chi$-bounds for $\chi$-bounded graph classes
	}
	\author{Tung H. Nguyen}
	\address{Mathematical Institute and Christ Church, University of Oxford, Oxford, UK}
	\email{\href{mailto:nguyent@maths.ox.ac.uk}{nguyent@maths.ox.ac.uk}
	}
	\author{Sang-il Oum}
	\address{Discrete Mathematics Group, Institute for Basic Science (IBS), Daejeon, South Korea}
	\address{Department of Mathematical Sciences, KAIST, Daejeon, South Korea}
	\email{\href{mailto:sangil@ibs.re.kr}{sangil@ibs.re.kr}}
	\thanks{The first author was supported by a Titchmarsh Research Fellowship and a Christ Church Research Centre Grant. The second author was supported by the Institute for Basic Science (IBS-R029-C1).}
	
	\begin{abstract}
		We prove that for every path $P$, 
		the class of graphs with no induced $P$
		and no induced four-cycle $C_4$ is linearly $\chi$-bounded. 
		More generally, we ask 
		for which pairs $\{T,H\}$ where $T$ is a forest and $H$ is a complete multipartite graph,
		every graph $G$ with no induced $T$ and no induced $H$ 
		has chromatic number at most $C \cdot R(\alpha(H),\omega(G)+1)$
		for some constant $C$ depending only on $T$ and $H$,
		where $R(\cdot,\cdot)$ denotes the usual Ramsey numbers.
		We show that this holds in the following two instances, which strengthen the case $T=P$ and $H=C_4$ mentioned above:
		\begin{enumerate}
			\item every component of $T$ is a broom and $H$ is complete multipartite; or
			
			\item $T$ is a forest and $H$ is complete bipartite.
		\end{enumerate} 
		
		These two unify and substantially extend a number of previous results 
		on linear and polynomial $\chi$-boundedness for various graph classes.
		
		For case (2), we also provide a new proof (with better bounds) of a recent result of Fox, Nenadov, and Pham on the existence of an induced copy of a fixed tree in a graph satisfying certain sparsity conditions.
	\end{abstract}
	\maketitle
	\section{Introduction}
	\label{sec:intro}

	Graphs in this paper are finite and simple. For a graph $G$, we let $\chi(G)$, $\alpha(G)$, and $\omega(G)$ denote its chromatic number, its stability (or independence) number, and its clique number, respectively.
	A graph $H$ is an \emph{induced subgraph} of $G$ if it is obtained from $G$ by removing vertices. We say that $G$ is \emph{$H$-free} if $G$ has no induced subgraph isomorphic to $H$. 
	For a set $\mac H$ of graphs, $G$ is \emph{$\mac H$-free} if $G$ is $H$-free for all $H\in\mac H$; and the \emph{$\mac H$-free class} is the class of all $\mac H$-free graphs.
	A graph class is \emph{hereditary} if it is closed under taking induced subgraphs. A hereditary class $\mac G$ is \emph{$\chi$-bounded} if there is a function $f\colon \mab N\to\mab R_{\ge0}$ such that $\chi(G)\le f(\omega(G))$ for all~$G\in\mac G$; and such a function~$f$ is a \emph{$\chi$-binding} function for~$\mac G$. The notion of $\chi$-boundedness was first introduced by Gy\'arf\'as~\cite{Gyarfas1975}; see~\cite{SS2018,Scott2023} for surveys of the area.
	
	For a $\chi$-bounded class $\mac G$, the \emph{optimal} $\chi$-binding function $f_{\mac G}$ is defined by 
	\[ f_{\mac G}(n):=\max (\chi(G):G\in\mac G,~\omega(G)=n).\] We say that $\mac G$ is \emph{polynomially $\chi$-bounded} if it admits a polynomial $\chi$-binding function, and \emph{linearly $\chi$-bounded} if it admits a linear $\chi$-binding function. In other words, $\mac G$ is linearly $\chi$-bounded if there exists $C=C(\mac G)\ge1$ such that $\chi(G)\le C\cdot\omega(G)$ for all $G\in\mac G$, or equivalently if its optimal $\chi$-binding function has linear growth.
	
	For an integer $k\ge1$, let $P_k$ denote the $k$-vertex path; and for $\eps>0$, we say that a graph $G$ is \emph{$(\eps,\chi)$-dense} if the set of nonneighbours of every vertex in $G$ induces a subgraph of $G$ with chromatic number less than $\eps\cdot\chi(G)$. The notion of $(\eps,\chi)$-dense graphs plays a central role in the chromatic density framework behind the first author's recent proof that the $P_5$-free class is polynomially $\chi$-bounded~\cite{Nguyen2025}.	
	We say that a hereditary class $\mac G$ is \emph{$\chi$-dense} if for every $\eps>0$, there exists $\delta=\delta(\mac G,\eps)>0$ such that every $G\in\mac G$ contains an $(\eps,\chi)$-dense induced subgraph $F$ with $\chi(F)\ge\delta\cdot\chi(G)$.
	In~\cite{Nguyen2025}, it is noted that $\chi$-dense classes are $\chi$-bounded with exponential $\chi$-binding functions; and so in general there exist $\chi$-bounded classes that are not $\chi$-dense, via constructions of Bria\'nski, Davies, and Walczak~\cite{BDW2024}.
	So, which $\chi$-bounded classes are $\chi$-dense?
	As a natural example, all linearly $\chi$-bounded classes are $\chi$-dense, because the complete graphs are clearly $(\eps,\chi)$-dense for any $\eps>0$.
	On the other hand, it is observed in~\cite{Nguyen2025} that the $C_4$-free subclass of every $\chi$-dense class is linearly $\chi$-bounded; and so the following more general phenomenon could be true:
	
	\begin{conjecture}[Nguyen~{\cite{Nguyen2025}}]
		\label{conj:c4chi}
		The $C_4$-free subclass of every $\chi$-bounded class is linearly $\chi$-bounded.
	\end{conjecture}
	
	In other words, \zcref{conj:c4chi} says that every hereditary $\chi$-bounded class not containing $C_4$ is linearly $\chi$-bounded.
	This is considerably stronger than a conjecture of Chudnovsky, Cook, Davies, and the second author~\cite[p. 70]{CCDO2023} that such classes are polynomially $\chi$-bounded.
	Chen and Xu~\cite{CX2026} recently obtained a weakening of \zcref{conj:c4chi}, that the $\{C_4,\text{bull}\}$-free and 
	$\{C_4,\text{hammer}\}$-free subclasses of every $\chi$-bounded class are linearly $\chi$-bounded.
	There are several well-known linear $\chi$-boundedness results supporting \zcref{conj:c4chi}; for example:
	\begin{itemize}
		\item Chudnovsky, Scott, and Seymour~\cite{CSS2015} proved that the class of graphs with no induced cycle of length at least $5$ is $\chi$-bounded. (They actually proved more, with $5$ replaced by any $\ell\ge3$.)
		Its $C_4$-free subclass is the class of chordal graphs, thus is a subclass of perfect graphs, and so is $\chi$-bounded with the identity $\chi$-binding function.
		
		\item Scott and Seymour~\cite{SS2014} proved that the class of graphs with no odd induced cycle is $\chi$-bounded. By a result of Conforti, Cournejols, and Vu{\v s}kovi\'c~\cite{CCV2004} (or more generally, by the strong perfect graph theorem~\cite{CRST2006}), every $C_4$-free graph with no odd induced cycle is perfect.
		
		\item Scott and Seymour~\cite{SS2015} also proved that the class of graphs with no induced cycle of length any specific residue is $\chi$-bounded; and so this property holds for the class of graphs with no induced cycle of length even and at least six. Chudnovsky and Seymour~\cite{CS2019} showed that its $C_4$-free subclass (also known as the class of \emph{even-hole-free} graphs) admits the linear $\chi$-binding function $x\mapsto 2x-1$.
	\end{itemize}
	
	Let us recall why the $C_4$-free subclass of every $\chi$-dense class is linearly $\chi$-bounded.
	Indeed, observe that every two nonadjacent vertices in a non-complete $(\frac13,\chi)$-dense graph $G$ have common neighbourhood inducing a subgraph of $G$ with chromatic number at least $\frac13\chi(G)$; and if $G$ is $C_4$-free then such a common neighbourhood is a clique.
	By the same logic, for each integer $m\ge1$, every $(\frac1{2m-1},\chi)$-dense and $\overline{mK_2}$-free graph $G$ satisfies $\omega(G)\ge\frac1{2m-1}\chi(G)$; and so the $\overline{mK_2}$-free subclass of every $\chi$-dense class is also linearly $\chi$-bounded.
	(Here, for a graph $H$, we use $\overline H$ to denote the complement graph of $H$ and $mH$ to denote the disjoint union of $m$ copies of $H$. Hence for every integer $s\ge1$, $\overline{mK_s}$ means the complete multipartite graph whose parts have size $s$.)
	Thus, perhaps the following generalisation of \zcref{conj:c4chi} holds as well:
	
	\begin{conjecture}
		\label{conj:cocktailchi}
		For every integer $m\ge1$, the $\overline{mK_2}$-free subclass of every $\chi$-bounded class is linearly $\chi$-bounded.
	\end{conjecture}
	
	For integers $s,w\ge1$, let $R(s,w)$ denote the standard Ramsey number, that is, the least integer $n\ge1$ such that every $n$-vertex graph has a stable set of size $s$ or a clique of size $w$. Every graph~$G$ with chromatic number at least $R(s,\omega(G)+1)$ necessarily contains a stable set of size $s$; and so it is not hard to see that if $G$ is in addition $(\eps,\chi)$-dense for some small $\eps>0$, then such a stable set has common neighbourhood inducing a subgraph of $G$ with chromatic number at least $(1-s\eps)\chi(G)$.
	Hence, for appropriate choices of $\eps$ depending on $m$ and $s$, the $\overline{mK_s}$-free subclass of every $\chi$-dense class $\mac G$ admits a $\chi$-binding function $x\mapsto C\cdot R(s,x+1)$, where $C=C(\mac G,m,s)\ge1$. 
	Again, such a general Ramsey-type extension could actually hold for all $\chi$-bounded classes, as follows.
	
	\begin{conjecture}
		\label{conj:cmpchi}
		For every $\chi$-bounded class $\mac G$ and every two integers $m,s\ge1$, there exists $C=C(\mac G,m,s)\ge1$ such that every $\overline{mK_s}$-free graph $G\in\mac G$ satisfies $\chi(G)\le C\cdot R(s,\omega(G)+1)$.
	\end{conjecture}
	
	Indeed, \zcref{conj:cocktailchi} is a special case of \zcref{conj:cmpchi} with $s=2$.
	
	A classical example of $\chi$-bounded classes is the $P_k$-free class for each $k\ge1$, as shown by Gy\'arf\'as~\cite{Gyarfas1987}. Let $f_k(x)$ denote the optimal $\chi$-binding function for the $P_k$-free class; then it is known~that
	\[\frac{R(\ceil{k/2},x+1)-1}{\ceil{k/2}-1}\le f_k(x)\le (k-2)^{x-1}\quad\text{for all $k\ge 4$}\]
	where the lower bound was observed by Gy\'arf\'as~\cite{Gyarfas1987} and the upper bound was due to Gravier, Ho\`ang, and Maffray~\cite{GHM2003}.
	In particular, $P_5$-free graphs are not linearly $\chi$-bounded because $R(3,w)\ge \Omega(w^2/\ln w)$ due to Kim~\cite{Kim1995}.
	It remains open whether the $P_k$-free  class is polynomially $\chi$-bounded for every $k\ge6$; and the first author~\cite{Nguyen2025} very recently verified this property for the $P_5$-free class.
	Provided these results and observations, it would be natural to see if the $P_k$-free classes are $\chi$-dense.
	However, we have been unable to decide this even for the $P_5$-free class; and so perhaps a more approachable goal would be to verify \zcref{conj:c4chi,conj:cocktailchi,conj:cmpchi} for graphs excluding a path.
	Our first main result achieves this goal, as follows.
	
	\begin{theorem}
		\label{thm:cocktailks}
		For every three integers $m\ge1$, $s\ge 2$, $k\ge 3$, every $\{P_k,\overline{m K_s}\}$-free graph $G$ satisfies 
		\[\chi(G)< (2s^{k-3})^{m-1}(4s+1)\cdot R(s,\omega(G)+1).\]
	\end{theorem}
	
	The proof method of \zcref{thm:cocktailks}, presented in \zcref{sec:stableset}, actually proves \zcref{conj:cmpchi} for the $(k,\ell)$-broom-free class for all $k\ge2$ and $\ell\ge1$.
	Here, the \emph{$(k,\ell)$-broom} is the graph obtained from the star $K_{1,\ell+1}$ by subdividing an edge exactly $k-2$ times (so the $(k,1)$-broom is $P_{k+1}$, see \zcref{fig:broom}); and as noted by Gy\'arf\'as~\cite{Gyarfas1987}, his proof that $P_k$-free graphs are $\chi$-bounded for all $k\ge1$ can be extended to reach the same conclusion for all $(k,\ell)$-broom-free classes.
	
	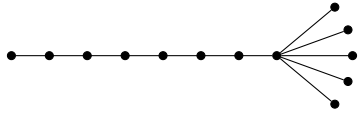
\begin{figure}
		\centering
		\begin{tikzpicture}
			\tikzstyle{v}=[circle, draw, solid, fill=black, inner sep=0pt, minimum width=3pt]
			\foreach \i in {0,1,2,3,4,5,6,7} {
				\node [v] (v\i) at (-0.5*\i,0) {};
			}
			\foreach \i in {1,2,3,4,5} {
				\node [v] (w\i) at (20*\i-60:1) {};
				\draw (v0)--(w\i);
			}
			\draw (v0)--(v7);
		\end{tikzpicture}
		\caption{The $(8,5)$-broom.}
		\label{fig:broom}
	\end{figure}
	
	\begin{theorem}
		\label{thm:cocktailbroomks}
		For every four integers $m\ge1$, $s\ge2$, $k\ge2$, and $\ell\ge1$, there exists $C=C(m,s,k,\ell)\ge1$ such that every $\{(k,\ell)\text{-broom},\overline{mK_s}\}$-free graph $G$ satisfies $\chi(G)\le C\cdot R(s,\omega(G)+1)$.
	\end{theorem}
	
	In what follows, for every integer $m\ge1$, let $K_m^-$ be the graph obtained by removing an edge from the $m$-vertex complete graph $K_m$; then $K_m^-$ is an induced subgraph of $\overline{(m+1)K_2}$.
	We remark that \zcref{thm:cocktailks,thm:cocktailbroomks} unify and substantially extend (up to a constant factor) a number of previous results in the literature on linear $\chi$-boundedness and polynomial $\chi$-boundedness (see the surveys~\cite{SR2019,CK2025} for more results on these topics); for instance:
	\begin{itemize}
		
		\item 
		Fouquet, Giakoumakis, Maire, and Thuillier~\cite{FGMT1995} proved that the $\{P_5,C_4\}$-free class admits a linear $\chi$-binding function $x\mapsto 3x/2$.
		This was later improved by Brause, Gei\ss{}er, and Schiermeyer~\cite{BGS2022} to $x\mapsto \ceil{(5x-1)/4}$. 
		
		\item Char and Karthick~\cite{CK2025a} proved that every $\{P_5,K_5^-\}$-free graph~$G$ is $\max\{7,\omega(G)\}$-colourable.
		
		\item Char and Karthick~\cite{CK2022} proved that the $\{P_5,\text{$4$-wheel}\}$-free class admits a linear $\chi$-binding function $x\mapsto \frac32x$, improving the previous result of Choudum, Karthick, and Shalu~\cite{CKS2007b}. The $4$-wheel is a graph obtained from $C_4$ by adding an additional vertex adjacent to all vertices of the~$C_4$.

		\item Wu, Li, and Li~\cite{WLL2025} proved that the $\{P_3\cup P_2, \text{$4$-wheel}\}$-free class admits a linear $\chi$-binding function $x\mapsto 2x$.
		
		\item
		Gaspers and Huang~\cite{GH2019} showed that the $\{P_6,C_4\}$-free class admits a linear $\chi$-binding function $x\mapsto 3x/2$, which was later improved to $x\mapsto \ceil{5x/4}$ by Karthick and Maffray~\cite{KM2019}.
		
		\item Cameron, Huang, Penev, and Sivaraman~\cite{CHPS2020} proved that the $\{P_7,C_4,C_5\}$-free class admits a linear $\chi$-binding function $x\mapsto 3x/2$, which was later optimised by Huang~\cite{Huang2024} to $x\mapsto \ceil{11x/9}$.
		\item Huang, Zhou, and Chang~\cite{HZC2026} proved that every $P_7$-free and even-hole-free graph~$G$ has chromatic number at most $\lceil \frac54\omega(G)\rceil$.
		
		\item Chudnovsky, Huang, Karthick, and Kaufmann~\cite{CHKK2021} proved that
		the $\{(3,2)\text{-broom},C_4\}$-free class admits a linear $\chi$-binding function $x\mapsto\ceil{3x/2}$; note that the $(3,2)$-broom is also known in the literature as the \emph{fork} or the \emph{chair}.
		
		\item Zhou, Li, Song, and Wu~\cite{ZLSW2024} proved that for every integer $\ell\ge2$, the $\{(4,\ell)\text{-broom},C_4\}$-free class admits a quadratic $\chi$-binding function. 
		
		\item Karthick and Mishra~\cite{KM2018a} 
		showed that the $\{P_6,K_4^-\}$-free class is linearly $\chi$-bounded; and later 
		Cameron, Huang, and  Merkel~\cite{CHM2021} and Goedgebeur, Huang, Ju, and Merkel~\cite{GHJM2023} improved its $\chi$-binding function; 
		note that $K_4^-$ is called the \emph{diamond} in the literature.
		
		\item Prashant, Francis Raj, and Gokulnath~\cite{PFG2023} proved that for every integer $m\ge1$, the $\{mK_2,K_5^-\}$-free class is linearly $\chi$-bounded.
		
		\item Dong, Xu, and Xu~\cite{DXX2022} claimed that 
		the $\{P_5,K_{2,3}\}$-free class 
		is $\chi$-bounded 
		by a $\chi$-binding function $x\mapsto 2x^2-x-3$,
		with an incorrect proof.
		Their gap was repaired by Ho\`ang~\cite{Hoang2026}.

		\item Liu, Schroeder, Wang, and Yu~\cite{LSWY2023} showed that the $\{(3,s)\text{-broom},K_{s,s}\}$-free class is $\chi$-bounded by a $\chi$-binding function $x\mapsto o(x^{s})$ for every $s\ge 3$.
		Note that \zcref{thm:cocktailbroomks} yields the $\chi$-binding function $x\mapsto o(x^{s-1})$ for this class, because Li, Rousseau, and Zang~\cite{LRZ2001} showed that 
		for every fixed integer~$s$, 
		$R(s,w)\le (1+o(1)) \frac{w^{s-1}}{\log^{s-2} w}$ for all sufficiently large~$w$ compared to~$s$, which improves an earlier result of Ajtai, Koml\'os, and Szemer\'edi~\cite{AKS1980}.
	\end{itemize}
	
	Chudnovsky, Scott, and Seymour~\cite{CSS2014b} proved that for every pair of graphs $H$ and~$J$, there exists an integer~$p$ such that 
	the vertex set of every $\{H,J\}$-free graph admits a partition into at most $p$ subsets, each inducing a subgraph that is edgeless, $X$-free, or $Y$-free for a component~$X$ of~$H$ or a component~$\overline{Y}$ of~$\overline{J}$.
	Combined with \zcref{thm:cocktailbroomks}, this implies that:
	
	\begin{theorem}
		For every two integers $m\ge1$ and $s\ge2$, and every forest $T$ whose components are brooms, the $\{T,\overline{mK_{s}}\}$-free class is $\chi$-bounded by a linear function of $R(s,\omega(G)+1)$.
	\end{theorem}
	
	To motivate our second main result, let us recall the famous Gy\'arf\'as--Sumner conjecture~\cite{Gyarfas1975,Sumner1981} that the class of graphs excluding any induced tree is $\chi$-bounded.
	Hajnal and R\"odl (see~\cite{GST1980,KP1994}) in the 1970s independently obtained the following well-known weakening:
	
	\begin{theorem}
		[Hajnal and R\"odl]
		\label{thm:hrtreekss}
		For every tree $T$ and every integer $s\ge1$, the $\{T,K_{s,s}\}$-free class is $\chi$-bounded.
	\end{theorem}
	
	In view of this theorem, \zcref{conj:cmpchi} with $m=2$ would imply:
	
	\begin{theorem}\label{thm:treekss}
		For every tree $T$ and every integer $s\ge1$,
		there exists $C=C(T,s)\ge1$ such that every $\{T,K_{s,s}\}$-free graph $G$ has chromatic number at most $C\cdot R(s,\omega(G)+1)$. 
	\end{theorem}
	
	By taking $s=2$, this result implies that the $\{T,C_4\}$-free class is linearly $\chi$-bounded, which also contains several aforesaid results on linear $\chi$-boundedness of graphs excluding a short path and $C_4$.
	\zcref{thm:treekss} is true, and actually follows from a recent result of Fox, Nenadov, and Pham~\cite{FNP2025} that a certain sparsity condition on a host graph forces an induced subgraph isomorphic to a fixed tree. Formally, for $c>0$ and an integer $t\ge2$, a graph $G$ is \emph{$(c,t)$-sparse} if for all (not necessarily disjoint) $A,B\subset V(G)$ with $\abs A,\abs B\ge t$, there are at most $(1-c)\abs A\abs B$ ordered pairs $(u,v)$ with $u\in A$, $v\in B$, and $uv\in E(G)$.
	Then the result of Fox, Nenadov, and Pham~\cite{FNP2025} says the following.
	
	\begin{theorem}[Fox, Nenadov, and Pham~{\cite[Theorem 1.7]{FNP2025}}]
		\label{thm:fnp25}
		For every tree $T$, every $c\in(0,\frac12)$, there exists $C=C(T,c)\ge1$ such that the following holds.
		For every integer $t\ge1$ and every two graphs $\Gamma$ and $G$ such that $G$ is a subgraph of $\Gamma$ and $\Gamma$ is $(c,t)$-sparse, either:
		\begin{itemize}
			\item $G$ has average degree less than $Ct$; or
			
			\item $G$ has a subgraph that is actually induced in $\Gamma$ and is isomorphic to $T$.
		\end{itemize}
	\end{theorem}
	
	To see how this implies \zcref{thm:treekss}, take $c=1/(4s)$, $t=2R(s,\omega(G)+1)$, and $G=\Gamma$. The details will be explained in \zcref{sec:altproof} where we discuss the connection of these results to degree-boundedness.
	
	The argument in~\cite{FNP2025} implies that one can take $C$ to be $c^{-2^{O(\abs{V(T)})}}$ in \zcref{thm:fnp25}, which is a polynomial in $1/c$ whose degree is exponential in $\abs{V(T)}$.
	Our second main objective in this paper is to give a new proof of \zcref{thm:fnp25} that implies substantially better dependence of the degree on~$T$, as follows.
	
	\begin{theorem}\label{thm:sparse-tree}
		For every tree $T$ of radius $h\ge0$, every $c\in(0,\frac12)$, every integer $t\ge1$, and every two graphs $\Gamma$ and $G$ such that $G$ is a subgraph of $\Gamma$ and $\Gamma$ is $(c,t)$-sparse, either:
		\begin{itemize}
			\item $G$ has average degree less than $(4/c)^{4h\abs{V(T)}}t$; or
			
			\item $G$ has a subgraph that is actually induced in $\Gamma$ and is isomorphic to $T$.
		\end{itemize}
	\end{theorem}
	
	(Here, recall that the \emph{radius} of the tree $T$ is the least $h\ge0$ such that there exists $v\in V(T)$ of distance at most $h$ to every vertex in $T$.)
	The proof of this result, presented in \zcref{sec:sparsetree}, involves the notion of locally injective homomorphisms which we find to be of independent interest.
	
	It would be interesting to improve the term $4h\abs{V(T)}$ to a fully linear function of $\abs{V(T)}$ in \zcref{thm:sparse-tree}; however we have been unable to decide this.
	It would also be very interesting to show that 
	for every forest $T$ and every two integers $m$, $s\ge2$, 
	$\{T,\overline{mK_s}\}$-free graphs are  $\chi$-bounded by a linear function of $R(s,\omega(G)+1)$, since this would unify \zcref{thm:treekss,thm:cocktailbroomks} and the Gy\'arf\'as--Sumner conjecture (by taking $m$ to be the clique number of the host graph).
	
	\subsection*{Notation}
	For an integer $k\ge1$, let $[k]:=\{1,2,\ldots,k\}$. For a graph $G$, let $\abs G:=\abs{V(G)}$. For $S\subset V(G)$, let $G[S]$ be the subgraph of $G$ obtained by removing all vertices in $V(G)\setminus S$, and let $G\setminus S:=G[V(G)\setminus S]$.
	For $v\in V(G)$, let $G\setminus v:=G\setminus\{v\}$, let $N_G(v):=\{u\in V(G)\setminus\{v\}:uv\in E(G)\}$, and let $N_G[v]:=N_G(v)\cup\{v\}$.
	For disjoint $A,B\subset V(G)$, we say that $A$ is \emph{anticomplete} to $B$, or $A$, $B$ are \emph{anticomplete}, if $G$ has no edge between $A$ and $B$.
	We say $A$ is \emph{complete} to $B$, or $A$, $B$ are \emph{complete}, if every vertex in $A$ is adjacent to all vertices of $B$.
	A \emph{stable} set is a set of pairwise nonadjacent vertices and a \emph{clique} is a set of pairwise adjacent vertices.
	For sets $X\subset A$ and $Y\subset B$, and a map $f\colon A\to B$, let $f(X):=\{f(x):x\in X\}$ and $f^{-1}(Y):=\{a\in A:f(a)\in Y\}$.

	\section{Stable sets with high-chromatic common neighbourhood in $P_k$-free graphs}
	\label{sec:stableset}
	
	This section presents the proof of \zcref{thm:cocktailks}.
	The proof method is via the following key fact.
	
	\begin{theorem}
		\label{thm:path2}
		Let $k\ge5$, $s\ge 2$, $q\ge 1$ be integers and let $G$ be a $P_k$-free graph with 
		\[\chi(G)\ge s^{k-3}\left(2q+7s\cdot R(s,\omega(G)+1)-7s\right).\]
		Then $G$ contains a stable set $S$ of size $s$  with $\chi(G[\bigcap_{v\in S} N_G(v)])> q$.
	\end{theorem}
	
	In other words, this result says that if $G$ is $P_k$-free and $\chi(G)$ is at least some large linear factor of $R(s,\omega(G)+1)$, then $G$ contains a stable set of size $s$ whose common neighbourhood has chromatic number linear in $\chi(G)$.
	Let us see now that \zcref{thm:path2} implies \zcref{thm:cocktailks}.
	
	\begin{proof}
		[Proof of \zcref{thm:cocktailks}, assuming \zcref{thm:path2}]
		We may assume that $k\ge 5$ because otherwise $G$ is perfect.
		Let
		$a:=2\cdot s^{k-3}$
		and $R:=R(s,\omega(G)+1)$.
		Let
		\(
		c:=\frac{7as(R-1)}{2(a-1)}
		\),
		so that
		\(
		-c=a\left(\frac72sR-\frac72s-c\right).
		\)
		We claim the following.
		
		\begin{clm}
			{\ref*{thm:cocktailks}.1}
			\label{claim:ineqks}
			$
			\chi(G)+c\le a^{m-1}(R-1+c)$ for all integers $m\ge1$.
		\end{clm}
		
		\begin{subproof}
			We proceed by induction on $m$.
			If $m=1$, then $G$ has no stable set of size $s$, and so $\chi(G)\le R-1$ by the definition of the Ramsey number.
			
			Thus we may assume that $m\ge2$.
			Let $S$ be a stable set of size $s$ in $G$. Since $\bigcap_{v\in S}N_G(v)$ is complete to $S$, 
			if 
			\(
			G\left[\bigcap_{v\in S}N_G(v)\right]
			\)
			contains an induced $\overline{(m-1)K_s}$, then combining this with $S$ would give an induced $\overline{mK_s}$ in $G$, a contradiction.
			Hence 
			\(
			G\left[\bigcap_{v\in S}N_G(v)\right]
			\)
			is $\overline{(m-1)K_s}$-free; and so the induction hypothesis implies
			\[
			\chi\left(G\left[\bigcap_{v\in S}N_G(v)\right]\right)\le a^{m-2}(R-1+c)-c.
			\]
			By \zcref{thm:path2} with $q:=\lfloor a^{m-2}(R-1+c)-c\rfloor $, we deduce that
			\[
			\chi(G)<a
			\left(q+\frac72sR-\frac72s\right)\le a^{m-1}(R-1+c)+a\left(\frac72sR-\frac72s-c\right)=a^{m-1}(R-1+c)-c.
			\]
			This proves \zcref{claim:ineqks}.
		\end{subproof}
		
		Since $k\ge5$ and $s\ge2$, we have $a=2\cdot s^{k-3}\ge 8$, and so 
		\[ R-1+c = \left(1+\frac{7as}{2(a-1)}\right)(R-1)
		\le \left(1+\left(7+\frac{7}{a-1}\right)\frac{s}{2}\right)(R-1)\le (1+4s)(R-1).\]
		Thus \zcref{claim:ineqks} yields
		\(
		\chi(G)\le a^{m-1}(R-1+c)< (2s^{k-3})^{m-1}(4s+1)R
		\).
		This proves \zcref{thm:cocktailks}.
	\end{proof}
	
	The rest of this section deals with the proof of \zcref{thm:path2}, which employs an analogue of the well-known `Gy\'arf\'as path' argument that was used by Gy\'arf\'as to prove that the $P_k$-free class is $\chi$-bounded for all $k\ge1$~\cite{Gyarfas1987}.
	We will actually follow the ideas behind the proof of~\cite[Lemma 2.3]{NSS2024}, but substantially modified to keep the numbers in shape.
	We would like to simultaneously maintain two vertex-disjoint substructures in $G$: an induced path $P$ and a highly connected subgraph $J$ with chromatic number linear in $\chi(G)$, such that the only vertex in $P$ with a neighbour in $J$ is an endpoint of $P$, and the portion of $J$ having no neighbour in $P$ still has large chromatic number. Each time we aim to lengthen $P$ by concatenating it with some induced path in $J$ and shrink $J$ essentially by a linear factor.
	In order to make this work, we apply the following result of the first author~\cite{Nguyen2024} on highly connected subgraphs with linear chromatic number, which numerically improves a fundamental theorem of Gir\~ao and Narayanan~\cite{GN2022}.
	In what follows, for an integer $a\ge0$, a graph $G$ is \emph{$a$-connected} if $\abs{V(G)}>a$ and there is no $S\subset V(G)$ with $\abs S<a$ such that $V(G)\setminus S$ admits a partition into two nonempty subsets that are disjoint and anticomplete in $G$.
	
	\begin{theorem}
		[Nguyen~{\cite[Proposition 4.4]{Nguyen2024}}]
		\label{thm:kappachi}
		For every integer $a\ge1$, every graph $G$ with $\chi(G)\ge 4a-1$ contains an $(a+1)$-connected induced subgraph $F$ with $\chi(F)\ge \chi(G)-2a+1$.
	\end{theorem}
	
	We are now ready to prove \zcref{thm:path2}, as follows.
	
	\begin{proof}[Proof of \zcref{thm:path2}]
		Suppose not. 
		Let $\omega=\omega(G)$.
		Let $a:=s(R(s,\omega+1)-1)$ and $r:=\frac{q+3a}{s-1}$, 
		let $f(1):=\frac1s(\chi(G)-q-2a)$; and for every $i\ge2$, let $f(i):=\frac1s (f(i-1)-(q+3a))$. Then $f(i)+r=s^{1-i}(f(1)+r)$ for all $i\ge 1$.
		The numbers $f(i)$ will be lower bounds on the chromatic number of the portion of $J$ with no neighbour in $P$ in the aforementioned sketch.
		
		We first use the hypothesis on $\chi(G)$ to deduce that:
		\begin{clm}
			{\ref*{thm:path2}.1}
			\label{claim:check2}
			$f(k-3)\ge q+4a$ and $f(k-2)>0$.
		\end{clm}
		\begin{subproof}
			Observe that 
			\[ f(1)+r=\frac1s(\chi(G)-q-2a)+\frac{q+3a}{s-1}
			= \frac1s\chi(G) + (q+2a)\left(\frac1{s-1}-\frac1s\right) +\frac{a}{s-1}
			>\frac1s\chi(G),\] 
			and so
			\begin{align*}
				f(k-3)=s^{4-k}(f(1)+r)-r> s^{3-k}\chi(G)-\frac{q+3a}{s-1}\ge s^{3-k}\chi(G)-(q+3a)\ge q+4a
			\end{align*}
			where the last inequality holds by the hypothesis. 
			Then $f(k-2)=\frac{1}{s}(f(k-3)-(q+3a))\ge \frac{a}{s}>0$.
			This proves \zcref{claim:check2}.
		\end{subproof}
		
		Now, \zcref{thm:kappachi} gives an $(a+1)$-connected induced subgraph $F$ of $G$ with
		$$\chi(F)\ge\chi(G)-2a+1\ge R(s,\omega+1).$$
		Thus $F$ has a stable set $S$ of size $s$. By our supposition, $\chi(G[\bigcap_{v\in S} N_F(v)])\le q$. Hence
		\[
		\sum_{v\in S} \chi(F\setminus N_F(v))\ge\chi\left(F\setminus \left(\bigcap_{v\in S} N_F(v)\right)\right)
		\ge \chi(F)-q> \chi(G)-q-2a
		\]
		and so there exists $v\in S$ 
		with $\chi(F\setminus N_F(v))>\frac1s(\chi(G)-q-2a)=f(1)$. Moreover, since $F$ is $(a+1)$\nobreakdash-connected, $v$ has more than $a=s(R(s,\omega+1)-1)$ neighbours in $V(F)$
		and $F\setminus v$ is $a$-connected.
		Thus there exists an integer $p\ge 1$ maximal such that $G$ contains an induced path $v_1\text-v_2\text-\cdots\text-v_p$ and an $R(s,\omega+1)$-connected induced subgraph $J$ of $G\setminus\{v_1,\ldots,v_p\}$ such that
		\begin{itemize}
			\item $\{v_1,\ldots,v_{p-1}\}$ is anticomplete to $V(J)$;
			
			\item $v_p$ has at least $R(s,\omega+1)$ neighbours in $V(J)$; and
			
			\item $\chi(J\setminus N_G(v_p))> f(p)$.
		\end{itemize}
		See \zcref{fig:box} for an illustration.
		
		Since $G$ is $P_k$-free and $f(k-2)>0$, we have $p\le k-3$ and so \zcref{claim:check2} yields
		\[\chi(J\setminus N_G(v_p))>f(p)\ge f(k-3)\ge 4a.\]
		Thus, \zcref{thm:kappachi} gives an $(a+1)$-connected induced subgraph $L$ of $J\setminus N_G(v_p)$ with 
		$$\chi(L)\ge\chi(J\setminus N_G(v_p))-2a+1> f(p)-2a\ge f(k-3)-2a.$$ 
		See \zcref{fig:box} for an illustration. Let
		\begin{align*}
			b&:=\frac1s(f(p)-q-2a)=\frac1s(f(p)-q-3a)+R(s,\omega+1)-1=f(p+1)+R(s,\omega+1)-1,\quad\text{and}\\ 
			Z&:=\{z\in V(J):\chi(L\setminus N_G(z))\le b\}.
		\end{align*}
		
		We claim that $\abs Z$ is small, as follows.
		
		\begin{clm}
			{\ref*{thm:path2}.2}
			\label{claim:main2}
			$\abs Z < R(s,\omega+1)$.
		\end{clm}
		\begin{subproof}
			Suppose that $\abs Z\ge R(s,\omega+1)$. Then $Z$ contains a stable set $S$ of size $s$. Thus $\chi(L\setminus(\bigcap_{v\in S} N_G(v)))\le bs$, which implies
			$$\chi\left(L\left[V(L)\cap \bigcap_{v\in S} N_G(v)\right]\right)\ge\chi(L)-bs> f(p)-2a-bs= q,$$
			contrary to the hypothesis. This proves \zcref{claim:main2}.
		\end{subproof}

		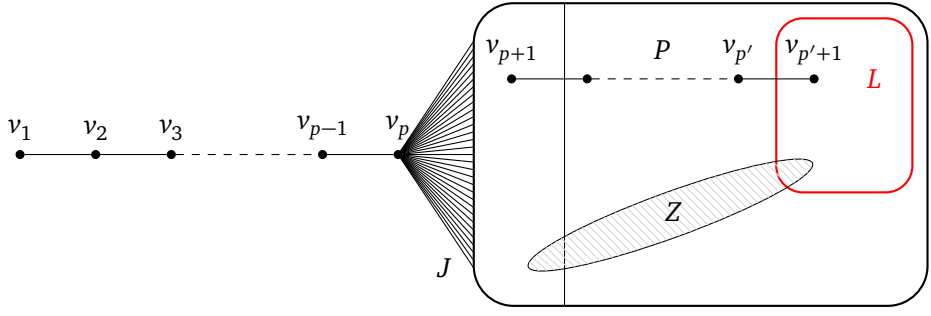
\begin{figure}
			\centering
			\begin{tikzpicture}
				\tikzset{v/.style={circle, draw, solid, fill=black, inner sep=0pt, minimum width=3pt}}
				
				\node [v, label=above:$v_1$] (v1) at (-6, 0) {};
				\node [v, label=above:$v_2$] (v2) at (-5, 0) {};
				\node [v, label=above:$v_3$] (v3) at (-4, 0) {};
				\node [v, label=above:$v_{p-1}$] (vs_1) at (-2, 0) {};
				\node [v, label=above:$v_p$] (vs) at (-1, 0) {};
				
				\draw (v1) -- (v3);
				\draw[dashed] (v3) -- (vs_1); 
				\draw (vs_1) -- (vs);
				
				\draw[thick, rounded corners=15pt] (0, -2) rectangle (6, 2);
				\node [label=left:$J$] at (0, -1.5) {};
				
				\draw (1.2, 2) -- (1.2, -2);
				
				\draw[red, thick, rounded corners=10pt] (4, -0.5) rectangle (5.8, 1.8);
				\node[red] at (5.3, 1) {$L$};
				
				\node [v, label=above:$v_{p+1}$] (vs+1) at (0.5, 1) {};
				\node [v] (mid) at (1.5, 1) {};
				\node [label=above:$P$] at (2.5,1){};
				\node [v, label=above:$v_{p'}$] (vsp) at (3.5,1){};
				\node [v, label=above:$v_{p'+1}$] (vsp+1) at (4.5, 1) {};
				
				\draw (vs+1) -- (mid);
				\draw [dashed] (mid)-- (vsp);
				\draw (vsp)--(vsp+1);
				
				\foreach \i in {0,1,2,...,30} {
					\draw (vs)--(0,0.1*\i-1.5);
				}
				
				\draw[rotate around={20:(2.6,-.8)}] (2.6,-.8) ellipse (2cm and 0.3cm);
				\fill[pattern=north west lines, pattern color=gray!30, rotate around={20:(2.6,-.8)}] (2.6,-.8) ellipse (2cm and 0.3cm);
				\node at (2.65, -.75) {$Z$};
				
			\end{tikzpicture}
			\caption{An illustration of the proof of \zcref{thm:path2}.}
			\label{fig:box}
		\end{figure}

		Now, since $J$ is $R(s,\omega+1)$-connected, $\abs{V(J)\cap N_G(v_p)}\ge R(s,\omega+1)$, and $\abs{V(L)}\ge R(s,\omega+1)$, \zcref{claim:main2} gives a path $P$ in $J$ between $V(J)\cap N_G(v_p)$ and $V(L)$ such that $V(P)\cap Z=\emptyset$.
		We may assume that $P$ is a shortest such path, and let $P=v_{p+1}\text-\cdots\text-v_{p'}\text-v_{p'+1}$ where $p'>p$. Since $L$ is $(a+1)$-connected, $v_{p'+1}$ has degree at least $a+1$ in $L$. Thus, if $v_{p'}$ has fewer than $R(s,\omega+1)$ neighbours in $V(L)$ then 
		$J':=L\setminus  N_G(v_{p'})$
		is $((a+1)-(R(s,\omega+1)-1))$-connected.
		Note that $(a+1)-(R(s,\omega+1)-1)\ge R(s,\omega+1)$ because $s\ge 2$. 
		Since $v_{p'+1}\notin Z$,  we have 
		\[
		\chi(J'\setminus N_G(v_{p'+1}))\ge \chi(L\setminus N_G(v_{p'+1}))-R(s,\omega+1)+1>b-R(s,\omega+1)+1=f(p+1)
		\] 
		and thus $v_1\text-\cdots\text-v_p\text-v_{p+1}\text-\cdots\text-v_{p'+1}$ and $J'$ would violate the maximality of $p$.
		Similarly, if $v_{p'}$ has at least $R(s,\omega+1)$ neighbours in $V(L)$ then $v_1\text-\cdots\text-v_p\text-v_{p+1}\text-\cdots\text-v_{p'}$ and $J':=L$ would contradict the maximality of~$p$ since $v_{p'}\notin Z$.
		This proves \zcref{thm:path2}.
	\end{proof}
	
	
	We remark that the same proof method also yields the following extension for $\{(k,\ell)\text{-broom},\overline{mK_s}\}$-free graphs.
	The proof can be found in \zcref{sec:broom}.
	
	\begin{theorem}
		\label{thm:broomks}
		Let $k\ge2$, $\ell\ge1$, $s\ge2$, $q\ge1$ be integers, and let $G$ be a $(k,\ell)\text{-broom}$-free graph.
		If
		\[
		\chi(G)\ge 7s^{k}\ell^{s+1}(q+R(s,\omega(G)+1)),
		\]
		then $G$ contains a stable set $S$ of size $s$ with
		\(
		\chi\left(G\left[\bigcap_{v\in S}N_G(v)\right]\right)>q
		\).
	\end{theorem}

	It is not hard to iterate \zcref{thm:broomks} to deduce \zcref{thm:cocktailbroomks}; we again omit its routine proof.

	\section{
		Degree-boundedness, $(c,t)$-sparseness, and $K_{s,s}$-freeness}\label{sec:altproof}
	
	In this section, we discuss the relations between degree-boundedness, $(c,t)$-sparseness, and $K_{s,s}$-freeness, and show a quick deduction of \zcref{thm:treekss} from \zcref{thm:fnp25} as promised in \zcref{sec:intro}. This section can also be viewed as a survey of known results in the literature.
	
	Recall that a hereditary class $\mac G$ of graphs is \emph{degree-bounded} 
	if there exists a function $g\colon\mab N\to\mab R_{\ge0}$ such that 
	for every $s\ge1$, every graph in $\mac G$ with no $K_{s,s}$ subgraph has average degree at most $g(s)$ (this is usually defined for minimum degree instead of average degree, but for hereditary classes the two definitions are equivalent; see~\cite{DM2025a}).
	If the function $g$ can be taken as a polynomial, then we call $\mac G$ \emph{polynomially degree-bounded}.
	It is not hard to see that $\mac G$ is (polynomially) $\chi$-bounded if it is (polynomially) degree-bounded and does not contain $K_{s,s}$ for some $s\ge1$.
	See~\cite{DM2025a} for a survey of degree-boundedness, which observed that a hereditary class is degree-bounded if and only if all of its bipartite $C_4$-free graphs have bounded average degree~\cite[Corollary 2.8]{DM2025a}.
	The two most well-known degree-bounded classes are:
	\begin{theorem}
		[Kierstead and Penrice~\cite{KP1994}]
		\label{thm:treedeg}
		For every tree $T$, the $T$-free class is degree-bounded.
	\end{theorem}
	\begin{theorem}
		[K\"uhn and Osthus~\cite{KO2004}]
		\label{thm:subsdeg}
		For every graph $H$, the class of graphs excluding all induced subdivisions of $H$ is degree-bounded.
	\end{theorem}
	
	(We remark that \zcref{thm:treedeg} implies \zcref{thm:hrtreekss}.)
	It was also recently shown by Gir\~ao and Hunter~\cite{GH2025} that degree-boundedness and its polynomial strengthening are equivalent:
	\begin{theorem}
		[Gir\~ao and Hunter~\cite{GH2025}]
		\label{thm:polydeg}
		A hereditary class is degree-bounded if and only if it is polynomially degree-bounded.
	\end{theorem}
	
	Let us recall that for $c>0$ and an integer $t\ge1$, a graph $G$ is \emph{$(c,t)$-sparse} if, for all subsets $A,B\subset V(G)$ (not necessarily disjoint) with $\abs A,\abs B\ge t$, we have
	\[e_G(A,B):=\abs{\{(a,b)\in A\times B:ab\in E(G)\}}\le (1-c)\abs A\abs B\]
	where the pairs in $A\cap B$ that form an edge of $G$ are counted twice in $e_G(A,B)$. Note that the class of $(c,t)$-sparse graphs is closed under taking subgraphs.
	Given this definition, let us say that a class~$\mac G$ of graphs is \emph{mildly sparse} if for every $c>0$, there exists $C=C(\mac G,c)\ge1$ such that for every integer~$t\ge1$ and every $G\in\mac G$, if $G$ is $(c,t)$-sparse then it has average degree less than $Ct$; and we say that~$\mac G$ is \emph{polynomially mildly sparse} if $C$ can be taken as a polynomial in $1/c$. Then \zcref{thm:fnp25} implies that the $T$-free class is mildly sparse for all trees $T$; and the proof in~\cite{FNP2025} shows that every such class is actually polynomially mildly sparse.
	
	It turns out that the notions of (polynomial) degree-boundedness and (polynomial) mild sparseness coincide, as follows.
	First, the latter implies the former (also for their polynomial versions) by taking $c=t^{-2}$ in the definition of mild sparseness.
	To see the other direction, we apply the following recent theorem of Hunter, Milojevi\'c, Sudakov, and Tomon~\cite{HMST2025a}.
	
	\begin{theorem}[Hunter, Milojevi\'c, Sudakov, and Tomon~{\cite{HMST2025a}}]\label{thm:reduce-to-c4free}
		For $k\ge1$ and $c>0$, there exists $C=C(k,c)\ge1$ such that for every $t\ge1$, every $(c,t)$-sparse graph with average degree at least $Ct$ contains
		a bipartite $C_4$-free induced subgraph of average degree at least $k$.
	\end{theorem}
	
	As mentioned, a hereditary class $\mac G$ is degree-bounded if and only if all of its bipartite $C_4$-free graphs have bounded average degree.
	Therefore \zcref{thm:reduce-to-c4free} implies that if $\mac G$ is degree-bounded then it is mildly sparse.
	As far as we know, the proof of \zcref{thm:reduce-to-c4free} in~\cite{HMST2025a} shows that $C$ can be taken as a polynomial in $1/c$ of degree $k^{O(k)}$; and so degree-boundedness actually implies polynomial mild sparseness.
	(However, combining this with \zcref{thm:treedeg} would not imply the dependence $C=c^{-2^{O(\abs T)}}$ in \zcref{thm:fnp25} that came from the proof of~\cite[Theorem 1.7]{FNP2025}, as mentioned in \zcref{sec:intro}.)
	In summary, we have the following extension of \zcref{thm:polydeg}, which  implies that \zcref{conj:c4chi}
	and \zcref{conj:cmpchi} for $m=2$ hold for $\chi$-bounded classes that are also degree-bounded:
	
	\begin{corollary}
		\label{cor:degbddctsparse}
		Let $\mac G$ be a hereditary class. Then the following are equivalent:
		\begin{itemize}
			\item $\mac G$ is degree-bounded;
			
			\item $\mac G$ is polynomially degree-bounded;
			
			\item $\mac G$ is mildly sparse; and
			
			\item $\mac G$ is polynomially mildly sparse.
		\end{itemize}
	\end{corollary}
	
	As promised in \zcref{sec:intro}, let us now see how \zcref{thm:fnp25} implies \zcref{thm:treekss}.
	We begin with the following lemma.
	
	\begin{lemma}
		\label{lem:kssparse}
		Let $s\ge2$ be an integer. Every $K_{s,s}$-free graph $G$ is $(1/(4s),\,2R(s,\omega(G)+1))$-sparse.
	\end{lemma}
	
	\begin{proof}
		Let $\omega:=\omega(G)$, $c:=1/(4s)$, and $t:=2R(s,\omega+1)$. Assume for contradiction that $G$ is not $(c,t)$-sparse. Then there exist $A,B\subset V(G)$ with $\abs A,\abs B\ge t$ and
		\begin{equation}
			\label{eq:dense}
			e_G(A,B)>(1-c)\abs A\abs B.
		\end{equation}
		We will produce an induced $K_{s,s}$ in~$G$, contradicting the hypothesis.
		
		Let $B^*:=\{b\in B:\abs{N_G(b)\cap A}\ge(1-2c)\abs A\}$. Counting ordered pairs gives
		\[(1-c)\abs A\abs B<e_G(A,B)=\sum_{b\in B}\abs{N_G(b)\cap A}\le\abs{B^*}\cdot\abs A+(\abs B-\abs{B^*})\cdot(1-2c)\abs A,\]
		which rearranges to $\abs{B^*}>\abs B/2\ge t/2=R(s,\omega+1)$.
		
		Since $G$ is $K_{\omega+1}$-free, so is $G[B^*]$; hence Ramsey's theorem yields a stable set $Y\subset B^*$ with $\abs Y=s$. By the defining property of~$B^*$, each $b\in Y$ has at most $2c\abs A$ non-neighbours in~$A$, and so
		\[A^*:=\bigcap_{b\in Y}(N_G(b)\cap A)\quad\text{satisfies}\quad\abs{A^*}\ge \abs A-s\cdot 2c\abs A=(1-2cs)\abs A=\tfrac12\abs A\ge R(s,\omega+1).\]
		Again, $G[A^*]$ is $K_{\omega+1}$-free, so Ramsey's theorem yields a stable set $X\subset A^*$ with $\abs X=s$. By construction, every vertex in~$X$ is adjacent to every vertex in~$Y$, and $X\cap Y=\emptyset$ because $X\subset A^*\subset N_G(b)$ while $b\in Y$ is not a neighbour of itself. Therefore $G[X\cup Y]$ is isomorphic to~$K_{s,s}$, an induced copy in~$G$, a contradiction.
		This proves \zcref{lem:kssparse}.
	\end{proof}
	
	We now apply \zcref{lem:kssparse} to deduce \zcref{thm:treekss} from \zcref{thm:fnp25} in the more general setting of mild sparseness, as follows.
	\begin{proposition}\label{prop:avgdeg}
		Let $\mac G$ be a mildly sparse class of graphs. Then for every integer $s\ge1$, there exists an integer $C=C(\mac G,s)\ge1$ such that every $K_{s,s}$-free graph $G\in\mac G$ has average degree less than $C\cdot R(s,\omega(G)+1)$.
		In particular, if $\mac G$ is, in addition, hereditary, then every $K_{s,s}$-free graph $G\in\mac G$ has chromatic number at most $C\cdot R(s,\omega(G)+1)$.
	\end{proposition}
	
	\begin{proof}
		For $s=1$, every $K_{s,s}$-free graph is edgeless, so $\de(G)=0$ and we may take $C=1$; hence assume $s\ge2$.
		Set $c_s:=1/(4s)$; then since $\mac G$ is mildly sparse, there exists $C'\ge1$ such that for every integer $t\ge1$ and every $G\in\mac G$, if $G$ is $(c_s,t)$-sparse then it has average degree less than $C't$.
		
		We claim that $C:=\ceil{2C'}$ suffices. To this end, let $G\in\mac G$ be $K_{s,s}$-free, and let $t:=2R(s,\omega(G)+1)$; then \zcref{lem:kssparse} implies that $G$ is $(c_s,t)$-sparse. Then $G$ has average degree less than $C't\le C\cdot R(s,\omega(G)+1)$, as required.
		
		To see the final implication, observe that every $K_{s,s}$-free graph $G\in\mac G$ has minimum degree less than $C\cdot R(s,\omega(G)+1)$, and so has degeneracy less than $C\cdot R(s,\omega(G)+1)$ since $\mac G$ is hereditary. This proves \zcref{prop:avgdeg}. 
	\end{proof}
	
	Another interesting application of \zcref{prop:avgdeg} is obtained by combining it with \zcref{thm:subsdeg,cor:degbddctsparse,lem:kssparse}.
	\begin{theorem}\label{thm:inducedsub}
		For every graph~$H$ and every integer $s\ge1$, 
		there exists $C=C(H,s)\ge1$ such that 
		every $K_{s,s}$-free graph $G$ with no induced subdivision of~$H$
		has chromatic number at most $C\cdot R(s,\omega(G)+1)$.
	\end{theorem}
	By \zcref{thm:inducedsub}, the $C_4$-free subclass of the class of graphs excluding all induced subdivisions of any given graph $H$ is linearly $\chi$-bounded.
	By applying this fact with $H=K_{2,3}$ and observing that every subdivision of $K_{2,3}$ contains an induced cycle of even length, we obtain another proof that the class of even-hole-free graphs is linearly $\chi$-bounded, which is implied by the result of Chudnovsky and Seymour~\cite{CS2019} mentioned in \zcref{sec:intro}.
	
	\section{Growing trees in $(c,t)$-sparse graphs}
	\label{sec:sparsetree}
	This section deals with \zcref{thm:sparse-tree}.
	To make the presentation clear, we will treat the case $G=\Gamma$ only; then the proof of the general case can be modified in the obvious way by making sure that `adjacent' means `adjacent in $G$' and `nonadjacent' means `nonadjacent in $\Gamma$' (we omit the details).
	Hence, our goal in this section is to prove the following.
	\begin{theorem}
		\label{thm:sparsetree}
		Let $F$ be a tree of radius $h\ge0$, let $c\in(0,\frac12)$, and let $t\ge1$ be an integer. Then every $(c,t)$-sparse graph $G$ with average degree at least $(4/c)^{4h\abs F}t$ contains an induced $F$.
	\end{theorem}

	To prove this result, we need several definitions.
	In what follows, a \emph{rooted tree} is a pair $(T,r)$ where $T$ is a tree and $r\in V(T)$; and we call $r$ the \emph{root} of $(T,r)$.
	The \emph{height} of a vertex in $(T,r)$ is its distance to $r$ in $T$; and the \emph{depth} of $(T,r)$ is the largest $h\ge0$ such that $(T,r)$ has a vertex of height~$h$.
	For $u,v\in V(T)$, we say that $v$ is a \emph{child} of $u$ and $u$ is the \emph{parent} of $v$ if $uv\in E(T)$ and the height of~$u$ is one less than the height of $v$, and that $v$ is a \emph{descendant} of $u$ if $u$ lies in the unique path from~$v$ to $r$ in $T$.
	An \emph{internal vertex} of $(T,r)$ is a vertex having children in $(T,r)$; and a \emph{leaf} of $(T,r)$ is a non-root vertex with no children.
	For a real number $d\ge0$ and an integer $h\ge0$, 
	we say that $(T,r)$ is \emph{$(d,h)$-wide} if its depth is $h$ and every vertex of height less than $h$ in $(T,r)$ has at least $d$ children.

	A homomorphism $\varphi$ from a graph $H$ to a graph $G$ is \emph{locally injective} if $\varphi\vert_{N_H(v)}$ is injective for all~$v\in V(H)$ (see~\cite{FK2008} for a survey on locally injective homomorphisms and related concepts). 
	For a rooted tree $(T,r)$, a \emph{skeleton} from $(T,r)$ to~$G$ is a locally injective homomorphism $\varphi$ from~$T$ to~$G$ such that 
	$\varphi$ maps 
	every root-to-leaf path of $(T,r)$ into an induced path in $G$; and it is called a \emph{$(d,h)$-skeleton}
	if $(T,r)$ is $(d,h)$-wide.
	We also say that a \emph{skeleton in $G$} is a skeleton from some rooted tree~$(T,r)$ to~$G$.
	If $r$ is a root of $T$, 
	then we call $\varphi(r)$ the \emph{root} of $\varphi(T)$.
	
	Skeletons are more relaxed than path-induced trees, which are fully injective homomorphisms with a rooted tree domain that map root-to-leaf paths to induced paths (see~\cite{DY2024,SSS2023,NSS2024a} for recent papers concerning path-induced trees).
	To explain, skeletons only require the neighbourhood of each vertex in the domain to be mapped injectively (hence `locally injective') and do not require vertices from different branches to have different images.
	This allows for more flexibility in our construction of induced trees in $(c,t)$-sparse graphs, which is conveniently divided into two steps:
	\begin{itemize}
		\item first, we will construct a sufficiently wide skeleton from a $(c,t)$-sparse graph with large average degree; and
		
		\item second, we will prove that every $(c,t)$-sparse graph with a sufficiently wide skeleton contains a copy of any given induced tree.
	\end{itemize}
	
	In the first step, our construction will be done by adding new \emph{roots} to existing skeletons to build a new one with larger depth; this explains why our argument differs significantly from the proof of~\cite[Theorem 1.7]{FNP2025} where trees are constructed by adding one \emph{leaf} at a time.
	In the second step, we use the $(c,t)$-sparseness condition to `clean' the overlaps and the cross-edges between images of different branches of the rooted tree domain.
	We are hopeful that locally injective homomorphisms (and skeletons in particular) will find more applications in $\chi$-boundedness and induced subgraphs.
	
	Now, in order to formalise the first step mentioned above, we define, for every $c\in(0,\frac12)$, the following function
	\[
	\Phi(c,0):=0,\quad \Phi(c,1):=1
	\]
	and
	\[
	\Phi(c,h+1):=(\Phi(c,h)+2)(4/c)^{h+1}
	\quad\text{for all }h\ge 1.
	\]
	By writing $\eps:=c/4$, the above recurrence unfolds to
	\[
	\Phi(c,h)
	=
	\eps^{-\binom{h+1}{2}}
	\left(\eps+
	2\sum_{i=2}^{h} \eps^{{i\choose2}}\right)
	\quad
	\text{for all $h\ge 1$,}
	\]
	where the empty sum is interpreted as~$0$.
	Since $\eps\in(0,\frac18)$, this yields
	\[\Phi(c,h)\le 4\eps^{1-\binom{h+1}{2}}=c(4/c)^{\binom{h+1}{2}}
	\quad\text{for all $h\ge1$.}\]
	For notational convenience, for a graph $G$ let $\de(G):=\abs{E(G)}/\abs G$ be one-half of the average degree of~$G$.
	Then our skeleton construction is done by the following lemma, which is proved in \zcref{subsec:skeleton}.
	
	\begin{lemma}
		\label{lem:skeleton}
		Let $h\ge0$ and $t\ge1$ be integers, let $d\ge t$ be real, and let $c\in(0,\frac12)$. 
		Then every $(c,t)$-sparse graph $G$ with $\de(G)\ge \Phi(c,h)\cdot d$ contains a $(d,h)$-skeleton.
	\end{lemma}
	
	Next, the following lemma shows how to extract an induced tree from a sufficiently wide skeleton as described in the second step above, which is proved in \zcref{subsec:turn}.
	
	\begin{lemma}
		\label{lem:turn}
		Let $(F,\rho)$ be a rooted tree with depth $h\ge0$, let $c\in(0,\frac12)$, and let $t\ge1$ be an integer. Then every $(c,t)$-sparse graph $G$ with a $((4/c)^{3h\abs F}t,h)$-skeleton contains an induced $F$.
	\end{lemma}

	Provided the last two lemmas, we are now ready to prove \zcref{thm:sparsetree}, as follows.
	
	\begin{proof}[Proof of \zcref{thm:sparsetree}, assuming \zcref{lem:skeleton,lem:turn}]
		We may assume that $\abs F\ge2$; thus $h\ge1$. Let $\eps:=c/4$, and suppose that $G$ has average degree at least $\eps^{-4h\abs F}t$. Then $\de(G)\ge \frac12\eps^{-4h\abs F}t$.
		Let $\rho\in V(F)$ be a vertex of distance at most $h$ to every other vertex in $F$.
		Since $h\le \abs F-1$, we have
		\[
		4h\abs F-3h\abs F-\binom{h+1}{2}
		\ge h\abs F-h\abs F/2\ge1.
		\]
		Because $\eps\in(0,\frac18)$, we have $\frac12\eps^{-1}\ge 4\eps=c\ge \eps^{{h+1\choose2}}\cdot\Phi(c,h)$, which yields
		\[
		\de(G)\ge \frac12\eps^{-4h\abs F}t
		\ge \frac12\eps^{-1-\binom{h+1}{2}-3h\abs F}t
		\ge \Phi(c,h)\cdot \eps^{-3h\abs F}t.
		\]
		
		Now, let $d:=\eps^{-3h\abs F}t$. \zcref{lem:skeleton} implies that there is a $(d,h)$-skeleton in $G$; and so \zcref{lem:turn} gives an induced $F$ in $G$. This proves \zcref{thm:sparse-tree}.
	\end{proof}
	
	\subsection{Some useful facts}
	To prove \zcref{lem:skeleton,lem:turn}, we require a couple of useful lemmas. The first one is a simple property of disjoint pairs of vertex subsets with not too high edge density.
	
	\begin{lemma}\label{lem:sparse}
		Let $G$ be a graph with disjoint $A,B\subset V(G)$
		such that $e_G(A,B)\le (1-c)\abs{A}\abs{B}$. Then the following hold:
		\begin{enumerate}
			\item $A$ has a vertex $v$ having at least $c\abs{B}$ non-neighbours in $B$.
			\item $A$ has at least $(c/2)\abs{A}$ vertices having at least $(c/2)\abs{B}$ non-neighbours in $B$.
		\end{enumerate}
	\end{lemma}
	\begin{proof}
		Since $\sum_{v\in A}\abs{N_G(v)\cap B}=e_G(A,B)\le(1-c)\abs{A}\abs{B}$, 
		there exists $v\in A$ with $\abs{N_G(v)\cap B}\le (1-c)\abs{B}$. 
		Thus $v$ has at least $c\abs{B}$ non-neighbours in $B$, proving the first fact.
		
		Now, for the second fact, let $A'=\{u\in A: \abs{N_G(u)\cap B}\le (1-\frac{c}{2})\abs{B}\}$.
		Then \[ (1-c)\abs{A}\abs{B}\ge e_G(A,B)\ge 
		\sum_{v\in A\setminus A'} \left(1-\frac{c}{2}\right)\abs{B} \] 
		and therefore $\abs{A\setminus A'}\le \frac{1-c}{1-(c/2)}\abs{A}$. 
		Hence $\abs{A'}\ge \frac{c/2}{1-(c/2)}\abs{A}\ge \frac{c}{2}\abs{A}$.
		This proves \zcref{lem:sparse}.
	\end{proof}
	
	Next, for a rooted tree $(T,r)$ and $\eps>0$, 
	an \emph{$\eps$-subtree} of $(T,r)$ is a rooted subtree $(T',r)$ of $(T,r)$ such that
	for all $u\in V(T)\cap V(T')$,
	the number of  children of $u$ in $(T',r)$
	is at least $\eps$ times the number of children of $u$ in $(T,r)$.
	For a skeleton $\varphi$ from $(T,r)$ to a graph~$G$, 
	a vertex $v\in V(G)\setminus \varphi(V(T))$ is \emph{$\eps$-good} for $\varphi$
	if there exists an $\eps$-subtree $(T',r)$ of $(T,r)$
	such that $\varphi(V(T'))\cap N_G[v]=\emptyset$.
	For $S\subset V(G)\setminus \varphi(V(T))$, we say that $S$ is \emph{$\eps$-good} for $\varphi$ if every $v\in S$ is $\eps$-good for $\varphi$.
	The following lemma, used in the proofs of \zcref{lem:skeleton,lem:turn}, says that given a sufficiently wide skeleton $\varphi$ in a~$(c,t)$-sparse graph $G$, and a sufficiently large set $X$ of nonneighbours of the root of $\varphi$ in $G$, one can shrink~$X$ by a polynomial of $1/c$ to a subset that is decently good for $\varphi$.

	\begin{lemma}\label{lem:sparse-shrink}
		Let $c\in(0,\frac12)$, 
		and let $h\ge0$, $t\ge 1$ be integers.
		Let $(T,r)$ be a $((4/c)^ht,h)$-wide rooted tree, and let $G$ be a $(c,t)$-sparse graph with a skeleton $\varphi$ from $(T,r)$ to $G$.
		Let $X\subseteq V(G)\setminus (\varphi(V(T))\cup  N_G(\varphi(r)))$,
		such that $\abs{X} \ge (4/c)^{h} t$.
		Then 
		there exists $Y\subset X$ with $\abs{Y}\ge (c/4)^h \abs{X}$ 
		such that $Y$ is $2(c/4)^{h+1}$-good for $\varphi$.
	\end{lemma}
	\begin{proof}
		Let $\eps:=c/4$.
		We proceed by induction on~$h$.
		The statement is trivial if $h=0$
		by taking $T'=T$ and $Y=X$.
		So we may assume that $h>0$ and the statement holds for $h-1$.
		Assume $G$, $\varphi$, and $X$ satisfy the hypothesis.
		Let $A:=\varphi(N_T(r))$.
		For each $v\in A$, let $X_v:=X\setminus N_G(v)$.
		By \zcref{lem:sparse} applied to $A$ and $X$,
		there exists $A_1\subseteq A$ such that $\abs{A_1}\ge2\eps\abs{A}$ 
		and 
		$\abs{X_v}\ge 2\eps\abs{X}$ for all $v\in A_1$.
		
		For each $v\in A_1$, let $\varphi^{-1}(v)\cap N_T(r)=\{r_v\}$ (note that $\varphi$ is locally injective), and let $(T_v,r_v)$ be the subtree of $(T,r)$ rooted at $r_v$.
		Then $\varphi\vert_{V(T_v)}$ is an $(\eps^{-h}t,h-1)$-skeleton from $(T_v,r_v)$ to $G$; and so the induction hypothesis applied to $\varphi\vert_{V(T_v)}$ gives $X_v'\subset X_v$ that is $2\eps^h$-good for $\varphi\vert_{V(T_v)}$ and satisfies
		\[\abs{X_v'}\ge \eps^{h-1}\abs{X_v}\ge 2 \eps^{h}\abs{X}.\]
		
		Now, let $H$ be an auxiliary bipartite graph with bipartition $(A_1,X)$ where $N_H(v)=X\setminus X_{v}'$ for all~$v\in A_1$.
		Then every vertex in $A_1$ has at least $2\eps^h\abs{X}$ non-neighbours in $H$.
		By \zcref{lem:sparse} (with~$2\eps^h$ replacing~$c$), there exists $Y\subset X$ with $\abs Y\ge\eps^h\abs X$  such that each $u\in Y$ has
		at least $\eps^{h}\abs{A_1}$ non-neighbours in $H$.
		Hence $Y$ 
		is $2\eps^{h}$-good for $\varphi\vert_{V(T_v)}$ for at least 
		$\eps^{h}\abs{A_1}$ distinct $v\in A_1$.
		
		It suffices to show that each $u\in Y$ is $2\eps^{h+1}$-good for $\varphi$.
		To this end, let $A_u$ be the set of vertices $v\in A_1$ such that
		$u$ is $2\eps^h$-good for $\varphi\vert_{V(T_v)}$;
		then 
		\[
		\abs{A_u}\ge \eps^{h}\abs{A_1}\ge 2\eps^{h+1}\abs{A}.
		\]
		For each $v\in A_u$, choose a $2\eps^h$-subtree $S_{u,v}$ of $T_v$ such that 
		\(
		\varphi(V(S_{u,v}))\cap N_G(u)=\emptyset.
		\)
		Let $T_u$ be the subtree of $T$ induced by $\{r\}$ and the union of $V(S_{u,v})$ for all $v\in A_u$.
		Since $u\in X\subseteq V(G)\setminus N_G(\varphi(r))$, we have $\varphi(r)\notin N_G(u)$.
		Also, for every vertex $x\in V(S_{u,v})$, the number of children of $x$ in $(T_u,r)$ is at least~$2\eps^{h}$ times the number of children of $x$ in $T_v$.
		Moreover, the root $r$ has at least $\abs{A_u}\ge 2\eps^{h+1}\abs{A}$ children in $(T_u,r)$.
		Therefore $(T_u,r)$ is a $2\eps^{h+1}$-subtree of $(T,r)$.
		By construction,
		\(
		\varphi(V(T_u))\cap N_G[u]=\emptyset,
		\)
		so $u$ is $2\eps^{h+1}$-good for $\varphi$.
		This proves \zcref{lem:sparse-shrink}.
	\end{proof}
	
	\subsection{Constructing skeletons}
	\label{subsec:skeleton}
	In this subsection we prove \zcref{lem:skeleton}.
	The proof is done by induction on $h$; that is, we aim to increase the depth $h$ of our skeleton one by one.
	We need one more definition. For a skeleton $\varphi$ from $(T,r)$ to $G$, we say that $u\in V(G)\setminus \varphi(V(T))$ is \emph{$\eps$-nice} for $\varphi$ if there is an $\eps$-subtree $(T',r)$ of $(T,r)$ such that $\varphi(V(T'))\cap N_G[u]=\{\varphi(r)\}$.
	To increase the depth $h$ by one we require the following lemma, whose proof is inspired by that of~\cite[Lemma 2.4]{SSS2023}.

	\begin{lemma}\label{lem:build-skeleton}
		Let $h\ge 1$, let $t\ge 1$ be an integer, let $d\ge t$ be real, and let $c\in(0,\frac12)$.
		Then for every $(c,t)$-sparse graph $G$, either:
		\begin{itemize}
			\item $G$ has an induced subgraph~$J$ with
			\(
			\de(J)\ge \de(G)-2(4/c)^{h+1}d
			\)
			such that there is no $((4/c)^{h+1}d,h)$-skeleton in $J$, or
			\item there is a $(d,h+1)$-skeleton in $G$.
		\end{itemize}
	\end{lemma}
	\begin{proof}
		Let $\eps:=c/4$, let $W:=\eps^{-(h+1)}d$, and assume that the first outcome does not hold.
		Let $(T,r)$ be the rooted tree of depth $h$ in which every vertex of height less than $h$ has exactly $\ceil W$ children, and let $X$ be the set of vertices $v\in V(G)$ for which $v$ is the root of some skeleton $\varphi_v$ from $(T,r)$ to $G$.
		Then there is no skeleton from $(T,r)$ to~$G\setminus X$; and so there is no $(W,h)$-skeleton in $G\setminus X$.
		Since the first outcome of the lemma does not hold, we have $\de(G\setminus  X)<\de(G)-2W$, which yields
		\[\abs{E(G\setminus  X)}=\de(G\setminus  X)\cdot \abs{G\setminus  X}
		<(\de(G)-2W)\cdot\abs G
		=\abs{E(G)}-2W\abs G.\]
		It follows that
		\[\sum_{v\in X}d_G(v)\ge \abs{E(G)}-\abs{E(G\setminus  X)}
		>2W\abs G.\]
		Let $Y:=\{v\in X:d_G(v)\ge \frac32 W\}$, then
		\[\sum_{v\in Y}d_G(v)=\sum_{v\in X}d_G(v)-\sum_{v\in X\setminus Y}d_G(v)>2W\abs G-\frac32 W\abs G=\frac12 W\abs G.\]
		
		We next claim that each $v\in Y$ admits many decently nice vertices for $\varphi_v$, as follows.
		
		\begin{claim}
			\label{claim:nice}
			For every $v\in Y$, at least $\frac14 \eps ^h d_G(v)$ vertices in $G$ are $2\eps ^{h+1}$-nice for $\varphi_v$.
		\end{claim}
		\begin{subproof}
			Let $A:=\varphi_v(N_{T}(r))\subset N_G(v)$; and let $B:=N_G(v)\setminus A$. Then 
			$\abs A=\ceil W \ge d\ge t$.
			Also, since $W>\eps^{-2}>64$ and $d_G(v)\ge\frac32W$, we have
			$\abs A=\ceil W<\frac98W\le\frac34d_G(v)$, and so
			$\abs B\ge \frac14d_G(v)\ge t$.
			Furthermore, $B\cap \varphi_v(V(T))=\emptyset$, 
			because otherwise $G[\varphi_v(V(L))]$ is not an induced path of $G$ for some root-to-leaf path $L$ of $T$.
			Since $G$ is $(c,t)$-sparse,
			there exists $P\subset A$ with $\abs P\ge\frac12c\abs A=2\eps\abs A$ such that every vertex in~$P$ has at least $\frac12c\abs B=2\eps\abs B$ nonneighbours in $B$, by \zcref{lem:sparse}.
			For each~$z\in P$, let 
			$z'$ be a child of $r$ in $T$ 
			such that $\varphi_v(z')=z$.
			Since $\varphi_v$ is locally injective,
			such a vertex~$z'$ of~$T$ is unique.
			Let $(T_z,z')$ be the subtree of~$(T,r)$ rooted at $z'$.
			Since $(T_z,z')$ is $(\eps^{-(h-1)}t,h-1)$-wide and $\abs{B\setminus N_G(z)}\ge2\eps\abs B\ge\eps^{-(h-1)}t$,
			by \zcref{lem:sparse-shrink}, there exists $B_z\subset B\setminus N_G(z)$ with $\abs{B_z}\ge \eps ^{h-1}\abs{B\setminus N_G(z)}\ge2\eps ^h\abs B$ such that $B_z$ is $2\eps ^h$-good for $\varphi_v\vert_{V(T_z)}$.
			Let $H$ be the auxiliary bipartite graph with bipartition $(P,B)$ where $N_H(z)=B\setminus B_z$ for all $z\in P$; then \zcref{lem:sparse} applied with parameter~$2\eps^h$ gives $Q\subset B$ with $\abs Q\ge \eps^h\abs B\ge \frac14\eps^h d_G(v)$ such that each vertex in $Q$ is non-adjacent in $H$ to at least $\eps^h\abs P\ge2\eps^{h+1}\abs A$ vertices of $P$.
			Then $Q$ is $2\eps^{h+1}$-nice for~$\varphi_v$. This proves \zcref{claim:nice}.
		\end{subproof}
		
		Now, \zcref{claim:nice} implies that the number of ordered pairs $(u,v)\in V(G)\times Y$ such that $u$ is $2\eps ^{h+1}$-nice for $\varphi_v$ is at least
		\[\frac14\eps ^h\sum_{v\in Y}d_G(v)
		>\frac{1}{8\eps}d\abs G>d\abs G
		\]
		and so there exist $u\in V(G)$ and at least $\ceil d$ vertices $v\in Y$ such that $u$ is $2\eps ^{h+1}$-nice for $\varphi_v$. Let $Z$ be a set of $\ceil d$ such vertices $v$.
		For each $v\in Z$, since $u$ is $2\eps ^{h+1}$-nice for $\varphi_v$ and $(T,r)$ is $(W,h)$-wide, there exists a $(d,h)$-wide rooted subtree $(T_v',r)$ of $(T,r)$
		such that $\varphi_v(V(T_v'))\cap N_G[u]=\{v\}$. Assembling~$u$ and the $(d,h)$-skeletons $(\varphi_v\vert_{V(T_v')}:v\in Z)$ then gives a $(d,h+1)$-skeleton in $G$ with root $u$. 
		Thus the second outcome of the lemma holds.
		This proves \zcref{lem:build-skeleton}.
	\end{proof}
	
	We are now ready to prove \zcref{lem:skeleton}, as follows.
	
	\begin{proof}[Proof of \zcref{lem:skeleton}]
		We prove the lemma by induction on~$h$.
		If $h=0$, then any vertex of $G$ gives a $(d,0)$-skeleton.
		If $h=1$, then some vertex of $G$ has degree at least $d$, giving a $(d,1)$-skeleton.
		So let $h\ge 1$, and assume that the lemma holds for~$h$.
		Let $G$ be a $(c,t)$-sparse graph with
		\(
		\de(G)\ge \Phi(c,h+1)\cdot d.
		\)
		By \zcref{lem:build-skeleton}, either:
		\begin{itemize}
			\item $G$ has an induced subgraph $J$ with $\de(J)\ge\de(G)-2(4/c)^{h+1}d$ such that there is no $((4/c)^{h+1}d,h)$-skeleton in $J$; or
			
			\item there is a $(d,h+1)$-skeleton in $G$.
		\end{itemize}
		If the second bullet holds then we are done; and so we may assume that the first bullet holds. Thus
		\[
		\de(J)
		\ge \de(G)-2(4/c)^{h+1}d
		\ge \Phi(c,h+1)\cdot d-2(4/c)^{h+1}d
		= \Phi(c,h)\cdot (4/c)^{h+1} d.
		\]
		Since $(4/c)^{h+1} d\ge d\ge t$, the induction hypothesis applied to $J$ with
		\(d':=(4/c)^{h+1} d\)
		gives a $(d',h)$-skeleton in $J$, contrary to the choice of~$J$.
		This proves \zcref{lem:skeleton}.
	\end{proof}
	
	\subsection{From skeletons to induced trees}
	\label{subsec:turn}
	In this subsection we complete the proof of \zcref{lem:turn}, which we briefly sketch as follows.
	We are given a skeleton $\varphi$ from some sufficiently wide rooted tree $(T,r)$ to $G$; and we would like to construct an injective homomorphism $\psi$ from $F$ to $T$ such that $\varphi\circ \psi$ is an isomorphism.
	We can first define $\psi(\rho):=r$, and aim to extend $\psi$ by adding to its domain one vertex of $F$ at a time.
	Assume that at some stage we have already defined $\psi$ on some subtree $F'$ of $F$ that contains $\rho$, and that $F'$ is not all of $F$ (else we are done).
	Thus there are vertices of $F'$ with at least one child in $(F,\rho)$ that is unmapped; let us call these vertices in $F'$ `unfinished'.
	To make sure that our construction works, we also assume that for each unfinished vertex $v$ there is a sufficiently wide subtree $(T_v,\psi(v))$ of $(T,r)$ with root $\psi(v)$, such that each vertex in $\varphi(\psi(V(F')))$ has no neighbour in $\varphi(V(T_v)\setminus\{\psi(v)\})$ for each unfinished $v$ (unless that vertex itself is $\varphi(\psi(v))$).
	Next we pick an unfinished $v$ with an unmapped child $z$, and aim to define $\psi(z)$ to be some appropriate vertex in $N_{T_v}(\psi(v))$.
	To do so, we first shrink $N_{T_v}(\psi(v))$ by a $\operatorname{poly}(c)$ factor to a subset $Q$ whose image under $\varphi$ is decently good for $\varphi\vert_{V(T_u)}$ for all unfinished $u$ different from~$v$ (using \zcref{lem:sparse-shrink});
	then pick $y\in \varphi(Q)$ and a sufficiently wide subtree $(T_v',\psi(v))$ of $(T_v,\psi(v))$ such that $y$ has no neighbour in $\varphi(V(T_v')\setminus\{\psi(v)\})$ (this part is similar to the proof of \zcref{claim:nice}).
	Then it remains to clean up the neighbours of $y$ in $(T_u,\psi(u))$ for each unfinished $u$ different from $v$, and define $\psi(z)$ to be the unique vertex in $\varphi^{-1}(y)\cap N_{T_v}(\psi(v))$.
	Let us now go into the details.
	
	\begin{proof}[Proof of \zcref{lem:turn}]
		Let $a:=3h\abs F$, let $\eps:=c/4$, and let $\varphi$ be a skeleton from an $(\eps^{-a}t,h)$-wide rooted tree $(T,r)$ to $G$ given by the hypothesis.
		For a subtree $F'$ of $F$ with $\rho\in V(F')$, we say that $u\in V(F')$ is \emph{unfinished} if it has a child in $V(F)\setminus V(F')$; thus every unfinished vertex in $V(F')$ is an internal vertex of $(F,\rho)$.
		Let $F'$ be a subtree of $F$, with $\rho\in V(F')$ and $\abs{F'}$ maximal, for which there exist an injective homomorphism $\psi$ from $F'$ to $T$ and, for each unfinished vertex $u\in V(F')$ of height~$h-j$ in $(F,\rho)$ for some $j\in[h]$,
		an $(\eps^{2h\abs{F'}-a}t,j)$-wide rooted subtree $(T_u,\psi(u))$ of $(T,r)$, such that:
		
		\begin{itemize}
			\item $\psi(\rho)=r$;
			
			\item $\varphi\circ\psi$ is an isomorphism from $F'$ to $G[\varphi(\psi(V(F')))]$; and
			
			\item for every unfinished vertex $u\in V(F')$, $\varphi(V(T_u)\setminus\{\psi(u)\})$ is disjoint from and anticomplete in~$G$ to $\varphi(\psi(V(F')\setminus\{u\}))$.
		\end{itemize}
		(Such a tree $F'$ exists, for example by taking $F'=F[\{\rho\}]$.)
		
		If $\abs{F'}=\abs F$ then $\varphi\circ\psi$ gives an induced copy of $F$ in $G$ and we are done; so let us suppose for a contradiction that $\abs{F'}<\abs F$.
		Then $\abs F\ge2$ and so $h\ge1$.
		Let $z\in V(F)\setminus V(F')$ have height minimal in $(F,\rho)$, and let $v$ be the parent of $z$ in $(F,\rho)$; then $v\in V(F')$ and $v$ is unfinished.
		Let $E$ be the set of unfinished vertices in $F'$ different from $v$, and let $j$ be the depth of $(T_v,\psi(v))$.
		
		For each $u\in E$, let $\varphi_u:=\varphi\vert_{V(T_u)}$.
		Let $D:=N_{T_v}(\psi(v))$; then $\varphi(D)$ is disjoint from and anticomplete to $\varphi(\psi(E))$ by the last bullet above.
		For each $x\in \varphi(D)$, the local injectivity of $\varphi$ yields $\abs{\varphi^{-1}(x)\cap D}=1$; so in what follows we also write $\varphi^{-1}(x)$ for the unique preimage of $x$ in $D$ under $\varphi$.
		
		To reach a contradiction, we aim to extend $\psi$ to $V(F')\cup\{z\}$. To do so, we claim that:
		
		\begin{clm}
			{\ref*{lem:turn}.1}
			\label{claim:getone}
			There exist a vertex $y\in \varphi(D)$ that is $2\eps^h$-good for $\varphi_u$ for all $u\in E$ and an $(\eps^{2h(\abs{F'}+1)-a}t,j)$-wide subtree $(T_v',\psi(v))$ of $(T_v,\psi(v))$ such that $\varphi(V(T_v')\setminus\{\psi(v)\})\cap N_G[y]=\emptyset$.
		\end{clm}
		\begin{subproof}
			Since $\abs F\ge\abs{F'}+1$, $\abs{F'}\ge\abs E+1$, and $\abs D\ge \eps^{2h\abs{F'}-a}t$, we have
			\[\eps^{h\abs E}\abs D\ge \eps^{h\abs E+2h\abs {F'}-a}t=\eps^{h\abs{E}+2h\abs {F'}-3h\abs F}t\ge\eps^{-4h}t.\]
			Note that for each $u\in E$, $\varphi(D)$ is disjoint from $\varphi(V(T_u)\setminus\{\psi(u)\})$ since the latter is anticomplete to $\varphi(\psi(v))$ in $G$, and also $\varphi(\psi(u))$ is disjoint from and anticomplete in $G$ to $\varphi(D)$.
			By the choice of~$z$, each $u\in E$ is not $\rho$; and so $T_u$ has depth at most~$h-1$.
			Hence, \zcref{lem:sparse-shrink} applied to $\varphi(D)$ versus $\varphi_u$ for each $u\in E$ gives $Q\subset D$ such that $\varphi(Q)$ is $2\eps^{h}$-good for $\varphi_u$ for all $u\in E$ and $\abs Q\ge\eps^{(h-1)\abs E}\abs D\ge\eps^{-3h}t$.
			Since $\varphi$ is locally injective, we can choose $Y\subset \varphi(Q)$ with $\abs Y=\ceil{\abs Q/2}$; let $X:=\varphi(D)\setminus Y$.
			Since $\eps<1/8$ and $h\ge1$, we have $\abs Y\ge \frac12\abs {Q}\ge\frac12\eps^{-3h}t\ge\eps^{-h-1} t$ and $\abs X\ge\abs D-\ceil{\frac12\abs D}=\floor{\frac12\abs D}\ge\frac14\abs D$.
			By \zcref{lem:sparse}, there exists $X'\subset X$ with $\abs{X'}\ge 2\eps \abs X$ such that every vertex in $X'$ has at least $2\eps \abs Y$ non-neighbours in $Y$.
			For each $x\in X'$, let $(T_x,\varphi^{-1}(x))$ be the subtree of $(T_v,\psi(v))$ rooted at $\varphi^{-1}(x)$ (so has depth $j-1\le h-1$), and let $\varphi_x:=\varphi\vert_{V(T_x)}$.
			Since $\varphi(\psi(v))$ is complete to $Y$ and anticomplete to $\varphi(V(T_x)\setminus\{\varphi^{-1}(x)\})$ in $G$, the set $Y\setminus N_G(x)$ has size at least $2\eps\abs Y\ge\eps\abs{Q}\ge\eps^{-h}t$ and is disjoint from $\varphi(V(T_x))\cup N_G(x)$.
			Hence \zcref{lem:sparse-shrink} gives $Y_x\subset Y\setminus N_G(x)$ with $\abs{Y_x}\ge  \eps ^{h-1}\cdot2\eps \abs Y=2 \eps ^{h}\abs Y$ such that $Y_x$ is $2 \eps ^{h}$-good for $\varphi_x$.
			Let $H$ be the bipartite graph with bipartition $(X',Y)$ where $N_H(x)=Y\setminus Y_x$ for all $x\in X'$; then \zcref{lem:sparse} applied with parameter $2\eps^h$ gives $y\in Y$ that is nonadjacent in $H$ to at least $ 2\eps ^{h}\abs {X'}\ge 4\eps^{h+1}\abs X\ge \eps^{h+1}\abs D$ vertices in $X'$.
			In other words, there exists $X_0\subset X'$ with (note that $h\ge1$)
			\[\abs{X_0}\ge \eps^{h+1}\abs D\ge \eps^{h+1+2h\abs{F'}-a}t\ge \eps^{2h(\abs{F'}+1)-a}t\]
			such that $y$ is $2\eps^h$-good for $\varphi_x$ for all $x\in X_0$.
			Then for each such $x$ there is a $2\eps^h$-subtree $(T_x',\varphi^{-1}(x))$ of $(T_x,\varphi^{-1}(x))$ with $\varphi(V(T_x'))\cap N_G[y]=\emptyset$. 
			Because $(T_x,\varphi^{-1}(x))$ is $(\eps^{2h\abs{F'}-a}t,j-1)$-wide, and
			\[2\eps^{h}\cdot \eps^{2h\abs{F'}-a}t\ge \eps^{h+2h\abs{F'}-a}t\ge \eps^{2h(\abs{F'}+1)-a}t,\]
			we see that $(T_x',\varphi^{-1}(x))$ is $(\eps^{2h(\abs{F'}+1)-a}t,j-1)$-wide for all $x\in X_0$.
			Thus, assembling $\psi(v)$ and $(T_x',\varphi^{-1}(x))$ for all $x\in X_0$ gives an $(\eps^{2h(\abs{F'}+1)-a}t,j)$-wide subtree $(T_v',\psi(v))$ of $(T_v,\psi(v))$ such that $\varphi(V(T_v')\setminus\{\psi(v)\})\cap N_G[y]=\emptyset$.
			This proves \zcref{claim:getone}.
		\end{subproof}
		
		Now, for each $u\in E$, let $j'$ be the depth of $(T_u,\psi(u))$. Since $y$ is $2 \eps ^{h}$-good for $\varphi_u$, there is a $2 \eps ^{h}$-subtree $(T_u',\psi(u))$ of $(T_u,\psi(u))$ such that $\varphi(V(T_u'))\cap N_G[y]=\emptyset$. Because
		\[2 \eps ^{h}\cdot  \eps ^{2h\abs{F'}-a}t
		\ge \eps^{2h(\abs{F'}+1)-a}t,\]
		we deduce that $(T_u',\psi(u))$ is $({\eps ^{2h(\abs{F'}+1)-a}t},j')$-wide.
		
		Finally, let $y':=\varphi^{-1}(y)$. Let $(T_{y'},y')$ be the subtree of $(T_v,\psi(v))$ rooted at $y'$; then $(T_{y'},y')$ is $(\eps^{2h\abs{F'}-a}t,j-1)$-wide and so is $(\eps^{2h(\abs{F'}+1)-a}t,j-1)$-wide.
		Because  $V(T_{y'})\subseteq V(T_v)\setminus\{\psi(v)\}$,
		$\varphi(\psi(V(F')\setminus\{v\}))$ is anticomplete to $\varphi(V(T_{y'}))$ in~$G$.
		Since $\varphi$ maps root-to-leaf paths in $(T,r)$ into induced paths in $G$, 
		$\varphi(\psi(v))$ has no neighbours in 
		$\varphi(V(T_{y'})\setminus\{y'\})$
		and therefore 
		$\varphi(\psi(V(F')))$ is anticomplete to $\varphi(V(T_{y'})\setminus\{y'\})$ in $G$.
		Thus, extending $\psi$ to $V(F')\cup\{z\}$ by defining $\psi(z):=y'$ would violate the maximality of $\abs{F'}$;
		note that $(T_v',\psi(v))$ is not needed if $v$ is no longer unfinished in $F[V(F')\cup\{z\}]$, and $(T_{y'},y')$ is not needed if $z$ is a leaf of $F$. This proves \zcref{lem:turn}.
	\end{proof}
	
	\section{Concluding remarks}
	\label{sec:oneinduced}
	We conclude the paper with the following question: for which graphs $H$ is the $H$-free subclass of every $\chi$-bounded hereditary class linearly $\chi$-bounded? A necessary condition is as follows.
	\begin{proposition}
		\label{prop:forcelin}
		Let $H$ be a graph such that for every $\chi$-bounded hereditary class $\mac G$, the $H$-free subclass of $\mac G$ is linearly $\chi$-bounded. Then $\overline H$ is a $P_5$-free forest.
	\end{proposition}
	\begin{proof}
		Suppose first that $\overline H$ contains an induced cycle $C$. For every sufficiently large $n$, a construction of Erd\H os~\cite{Erdos1959} gives an $n$-vertex graph $G_n$ with no cycle of length at most $\abs C$ and $\alpha(G_n)\le n^{1-1/(2\abs C)}$. In particular $G_n$ is $\{C_3,C_4,\overline H\}$-free; and so $\overline{G_n}$ is $\{3K_1,2K_2,H\}$-free, which yields $\omega(\overline{G_n})=\alpha(G_n)\le n^{1-1/(2\abs C)}\le\alpha(\overline{G_n})\chi(\overline{G_n})^{1-1/(2\abs C)}\le 2\chi(\overline{G_n})^{1-1/(2\abs C)}$. Thus the $\{3K_1,2K_2,H\}$-free class is not linearly $\chi$-bounded, a contradiction. Hence $\overline H$ is a forest.
		
		Now, suppose that $\overline H$ contains an induced $P_5$. Let $\mac G$ be the substitution closure of the induced subgraphs of the five-cycle $C_5$.
		Then $\mac G$ is polynomially $\chi$-bounded by a theorem of Chudnovsky, Penev, Scott, and Trotignon~\cite{CPST2013}; and every graph in $\mac G$ is $\overline{P_5}$-free (so $H$-free). 
		However, $\mac G$ contains the $k$-th blowup of $C_5$ for each $k\ge1$, which has clique number $2^k$, stability number $2^k$, and chromatic number at least $(5/2)^k=(5/4)^k\cdot2^k$. 
		Taking $k$ arbitrarily large implies that $\mac G$ is not linearly $\chi$-bounded, a contradiction. This proves \zcref{prop:forcelin}.
	\end{proof}
	
	It could be true that \zcref{prop:forcelin} is also sufficient; that is, if $\overline H$ is a $P_5$-free forest, then the $H$-free subclass of every $\chi$-bounded class is linearly $\chi$-bounded.
	If true, this would imply \zcref{conj:cocktailchi}.
	In addition to the $C_4$-free case (\zcref{conj:c4chi}), the other natural special instance of \zcref{conj:cocktailchi} is the $K_4^-$-free case (recall that $K_4^-$ is also known as the diamond in the literature).
	It would be interesting to decide whether the $K_4^-$-free subclass of every $\chi$-bounded class is linearly $\chi$-bounded; note that it is not even known if polynomial $\chi$-boundedness holds, as remarked in~\cite{CCDO2023}.
	
	\section*{Acknowledgments}
	This work was initiated while the first author visited the second author at the Discrete Mathematics Group, Institute for Basic Science (IBS DIMAG), and was partly carried out when both authors were participating in the 2026 MATRIX Workshop Global Structure and Geometry of Graphs (2026 GSGG). The first author would like to thank the second author and the participants of IBS DIMAG for their hospitality, and both authors would like to thank the organisers and members of 2026 GSGG for creating an encouraging working environment.
	
	For the purpose of open access, the authors have applied a CC BY public copyright licence to
	any author accepted manuscript arising from this submission.
	
	{\setstretch{1}
		\bibliographystyle{abbrv}

	}

	\appendix

	\section{Brooms}\label{sec:broom}
	
	In this appendix we give a proof of \zcref{thm:broomks}. We begin with:
	
	\begin{lemma}
		\label{lem:stablechi}
		Let $q\ge1$ and $s\ge2$ be integers. Then every graph $G$ with
		$$\chi(G)\ge {\alpha(G)\choose s}\cdot q+\alpha(G)^{s-1}\cdot R(s,\omega(G)+1)$$ 
		contains a stable set $I$ with $\abs I=s$ and $\chi(G[\bigcap_{v\in I}N_G(v)])>q$.
		In particular, for every integer $a\ge1$, if $\abs G\ge a{a\choose s}\cdot q+a^s\cdot R(s,\omega(G)+1)$ then either $\alpha(G)>a$ or $G$ contains a stable set $I$ with $\abs I=s$ and $\chi(G[\bigcap_{v\in I}N_G(v)])>q$.
	\end{lemma}
	\begin{proof}
		Let $R:=R(s,\omega(G)+1)$; and suppose for a contradiction that the conclusion does not hold.
		Let $S$ be a maximum stable set of $G$; then $\abs S\ge s$ since $\abs G\ge\chi(G)\ge R$, and every vertex in $G\setminus S$ has a neighbour in $S$.
		For every $I\subset S$ with $\abs I=s$, let $N_I:=\{v\in V(G)\setminus S:I\subset N_G(v)\}$; then $\chi(G[N_I])\le q$ by our supposition.
		For every $J\subset S$ with $1\le\abs J<s$, let $N_J:=\{v\in V(G)\setminus S:J=N_G(v)\cap S\}$; then $\alpha(G[N_J])\le\abs J<s$ (by the maximality of $S$) and $\omega(G[N_J])<\omega(G)$, which yields $\abs{N_J}<R(s,\omega(G))\le R$.
		In particular $\abs{N_{\{v\}}\cup\{v\}}\le R$ for all $v\in S$. Therefore, since $s\ge2$ and $V(G)\setminus S=\bigcup_{J\subset S,1\le\abs J\le s}N_J$, we obtain
		\begin{align*}
			\chi(G)&\le\sum_{I\subset S,\abs I=s}\chi(G[N_I])+\sum_{J\subset S,2\le\abs J<s}\abs{N_J}+\sum_{v\in S}\,\abs{N_{\{v\}}\cup\{v\}}\\
			&\le {\abs S\choose s}\cdot q+\sum_{j=2}^{s-1}{\abs S\choose j}\cdot R+\abs S\cdot R
			\le{\abs S\choose s}\cdot q+\abs S^{s-1}\cdot R,
		\end{align*}
		a contradiction. This proves the first statement.
		
		Now, observe that if $\alpha(G)\le a$ then $\chi(G)\ge\abs G/\alpha(G)\ge {\alpha(G)\choose s}\cdot q+\alpha(G)^{s-1}\cdot R(s,\omega(G)+1)$.
		This proves the second statement.
	\end{proof}
	
	We are now ready to prove \zcref{thm:broomks}, as follows.
	
	\begin{proof}
		[Proof of \zcref{thm:broomks}]
		Suppose not.
		Let $\omega:=\omega(G)$,
		$R:=R(s,\omega+1)$,
		and $B:=\ell{\ell\choose s}\cdot q+\ell^s\cdot R$.
		Let
		$a:=s(B-1)$ and $r:=\frac{q+3a}{s-1}$, 
		and define
		\[
		f(1):=\frac1s(\chi(G)-q-2a),\qquad f(i):=\frac1s(f(i-1)-(q+3a))\quad\text{for all }i\ge2.
		\]
		Then
		$f(i)=s^{1-i}(f(1)+r)-r$ for all $i\ge1$.
		
		\begin{clm}
			{\ref*{thm:broomks}.1}
			\label{claim:checkbroomks}
			$f(k-1)\ge 4a$.
		\end{clm}
		\begin{subproof}
			Observe that
			\(
			f(1)+r=\frac1s(\chi(G)-(q+2a))+\frac{q+3a}{s-1}
			= \frac1s \chi(G)
			+ \frac{q+3a}{s(s-1)}
			+ \frac{a}{s}
			\ge \frac1s\chi(G)
			\),
			and so
			\[
			f(k-1)=s^{2-k}(f(1)+r)-r\ge s^{1-k}\chi(G)-\frac{q+3a}{s-1}\ge s^{1-k}\chi(G)-(q+3a)\ge 4a
			\]
			where the last inequality holds by the hypothesis.
			This proves \zcref{claim:checkbroomks}.
		\end{subproof}
		
		Now, \zcref{thm:kappachi} gives an $(a+1)$-connected induced subgraph $F$ of $G$ with
		\(
		\chi(F)\ge \chi(G)-2a+1\ge B
		\).
		Thus $F$ contains a stable set $S$ of size $s$. By our supposition,
		\(
		\chi\left(G\left[\bigcap_{v\in S}N_F(v)\right]\right)\le q
		\).
		Hence
		\(
		\sum_{v\in S}\chi(F\setminus N_F(v))\ge \chi\left(F\setminus \bigcap_{v\in S}N_F(v)\right)\ge \chi(F)-q>\chi(G)-q-2a
		\),
		and so there exists $v\in S$ such that
		\(
		\chi(F\setminus N_F(v))>\frac1s(\chi(G)-q-2a)=f(1)
		\).
		Since $F$ is $(a+1)$-connected, $v$ has more than $a=s(B-1)\ge B$ neighbours in $V(F)$, and $F\setminus v$ is $a$-connected.
		
		Thus there exists an integer $p\ge1$ maximal such that $G$ contains an induced path $v_1\text-v_2\text-\cdots\text-v_p$ and a $B$-connected induced subgraph $J$ of $G\setminus\{v_1,\ldots,v_p\}$ such that:
		\begin{itemize}
			\item $\{v_1,\ldots,v_{p-1}\}$ is anticomplete to $V(J)$,
			\item $v_p$ has at least $B$ neighbours in $V(J)$, and
			\item $\chi(J\setminus N_G(v_p))>f(p)$.
		\end{itemize}
		
		Since $\omega(G[V(J)\cap N_G(v_p)])\le \omega$ and $\abs{V(J)\cap N_G(v_p)}\ge B$, \zcref{lem:stablechi} applied to $G[V(J)\cap N_G(v_p)]$ yields a stable set of size $\ell$. Since $G$ is $(k,\ell)$-broom-free, we have $p\le k-1$, and so \zcref{claim:checkbroomks} yields
		\[
		\chi(J\setminus N_G(v_p))>f(p)\ge f(k-1)\ge 4a.
		\]
		Thus \zcref{thm:kappachi} gives an $(a+1)$-connected induced subgraph $L$ of $J\setminus N_G(v_p)$ with
		\(
		\chi(L)\ge \chi(J\setminus N_G(v_p))-2a+1>f(p)-2a\ge f(k-1)-2a
		\).
		Let
		\begin{align*}
			b&:=\frac1s(f(p)-q-2a)=\frac1s(f(p)-(q+3a))+B-1=f(p+1)+B-1,\quad\text{and}\\
			Z&:=\{z\in V(J):\chi(L\setminus N_G(z))\le b\}.
		\end{align*}
		
		\begin{clm}
			{\ref*{thm:broomks}.2}
			\label{claim:mainbroomks}
			$\abs Z<R$.
		\end{clm}
		\begin{subproof}
			Suppose that $\abs Z\ge R$. Then $Z$ contains a stable set $S$ of size $s$. Thus
			\(
			\chi\left(L\setminus \bigcap_{v\in S}N_G(v)\right)\le bs
			\),
			which implies
			\(
			\chi\left(L\left[V(L)\cap \bigcap_{v\in S}N_G(v)\right]\right)\ge \chi(L)-bs>f(p)-2a-bs=q
			\),
			contrary to the hypothesis. This proves \zcref{claim:mainbroomks}.
		\end{subproof}
		
		Since $J$ is $B$-connected, $\abs{V(J)\cap N_G(v_p)}\ge B$, and $\chi(L)>2a\ge B\ge R$, \zcref{claim:mainbroomks} gives a path $P$ in~$J$ between $V(J)\cap N_G(v_p)$ and $V(L)$ such that $V(P)\cap Z=\emptyset$.
		We may assume that $P$ is a shortest such path, and let
		\(
		P=v_{p+1}\text-\cdots\text-v_{p'}\text-v_{p'+1}
		\)
		where $p'>p$. Since $L$ is $(a+1)$-connected, $v_{p'+1}$ has degree at least $a+1$ in $L$.
		
		If $v_{p'}$ has fewer than $B$ neighbours in $V(L)$, then
		\(
		J':=L\setminus N_G(v_{p'})
		\)
		is $((a+1)-(B-1))$-connected. Since $a=s(B-1)$ and $s\ge2$, we have $(a+1)-(B-1)\ge B$. 
		Also, $v_{p'+1}$ has at least $(a+1)-(B-1)\ge B$ neighbours in $J'$.
		Moreover, since $v_{p'+1}\notin Z$, we have
		\[
		\chi(J'\setminus N_G(v_{p'+1}))\ge \chi(L\setminus N_G(v_{p'+1}))-B+1>b-B+1=f(p+1)\ge f(p'+1),
		\]
		and so $v_1\text-\cdots\text-v_p\text-v_{p+1}\text-\cdots\text-v_{p'+1}$ and $J'$ contradict the maximality of $p$.
		
		If $v_{p'}$ has at least $B$ neighbours in $V(L)$, then 
		\(
		\chi(L\setminus N_G(v_{p'}))
		>b \ge f(p+1)\ge f(p')
		\)
		because $v_{p'}\notin Z$
		and so 		
		$v_1\text-\cdots\text-v_p\text-v_{p+1}\text-\cdots\text-v_{p'}$ and $J':=L$ violate the maximality of $p$.
		This proves \zcref{thm:broomks}.
	\end{proof}
\end{document}